\documentclass[openacc]{rsproca_new}%%%%where rsproca is the template name

\newtheorem{theorem}{\bf Theorem}[section]

\newcommand{\rmd}{\mathrm{d}}
\newcommand{\rmi}{\mathrm{i}}
\newcommand{\rme}{\mathrm{e}}
\newcommand{\im}{\mathrm{Im}}
\newcommand{\re}{\mathrm{Re}}
\newcommand{\pd}{\partial}
\newcommand{\zb}{\bar{z}}
\newcommand{\mymat}[1]{\begin{pmatrix} #1\end{pmatrix}}
\newcommand{\vr}{\mathbf{r}}

\usepackage{color}
\usepackage{enumerate}
\usepackage[normalem]{ulem}

\begin{document}

\title{
Construction of exact minimal parking garages:
Nonlinear helical motifs in optimally packed lamellar structures}

\author{Luiz C. B. da Silva 
and Efi Efrati}

\address{Department of Physics of Complex Systems,
 Weizmann Institute of Science, Rehovot 7610001, Israel}

\subject{Geometry, applied mathematics, mathematical physics}

\keywords{Minimal surface, holomorphic representation, Helical, Twist grain boundary}

\corres{
\email{luiz.da-silva@weizmann.ac.il}
\email{efi.efrati@weizmann.ac.il}
}

\begin{abstract}
Minimal surfaces arise as energy minimizers for fluid membranes and are thus found in a variety of biological systems. The tight lamellar structures of the endoplasmic reticulum and plant thylakoids are composed of such minimal surfaces in which right and left handed helical motifs are embedded in stoichiometry suggesting global pitch balance. So far, the analytical treatment of helical motifs in minimal surfaces was limited to the small-slope approximation where motifs are represented by the graph of harmonic functions. However, in most biologically and physically relevant regimes the inter-motif separation is comparable with its pitch, and thus this approximation fails. Here, we present a recipe for constructing exact minimal surfaces with an arbitrary distribution of helical motifs, showing that any harmonic graph can be deformed into a minimal surface by exploiting lateral displacements only. We analyze in detail pairs of motifs of the similar and of opposite handedness and also an infinite chain of identical motifs with similar or alternating handedness. Last, we study the second variation of the area functional for collections of helical motifs with asymptotic helicoidal structure and show that in this subclass of minimal surfaces stability requires that the collection of motifs is pitch balanced.  
\end{abstract}

\begin{fmtext}
 \end{fmtext} \maketitle
\section{Introduction}

From soap films and liquid crystals to plant thylakoids and the endoplasmic reticulum, minimal surfaces arise as the ground state of a variety of manmade and naturally occurring membranes and lamellar structures. Minimal surfaces minimize the area of a surface that passes through a given boundary \cite{Rado1933}, and thus arise naturally in cases where surface tension is dominant. However, minimal surfaces also arise in systems dominated by membrane bending energy as they constitute trivial critical points of the Helfrich free energy \cite{HelfrichZfN1973}.

Recently, helical motifs were discovered in the lamellar structures of the endoplasmic reticulum \cite{TerasakiCell2013} and plant and cyanobacteria thylakoids \cite{MustardyPC2008,LibertonPP2011}. Both right and left handed motifs were observed in both systems, in stoichiometry that suggested global pitch balance. While in the endoplasmic reticulum the right and left handed helical motifs were mirror images of one another and appeared in equal amounts, in the thylakoid the right and left handed motifs differed in their pitch, core radius, and density yet preserved pitch balance on average \cite{BussiPNAS2019}. 

To advance our understanding of these systems we require the ability to construct minimal surfaces in which the appropriate helical motifs are embedded. Twist grain boundaries, composed of infinitely many helical motifs of a single handedness embedded along a line in a minimal surface, allow an exact formulation and are well understood \cite{KamienPRL1999,SantangeloPRL2006,MatsumotoIF2017}. In contrast, embedding finitely many motifs in a minimal surface, as well as combining motifs of different pitch values remains a challenge. Coming to address these structures more explicitly Guven \emph{et al} \cite{GuvenPRL2014} considered the small slope approximation for minimal surfaces. The main advantage of this approach is that minimal surfaces are obtained as harmonic graphs (i.e. surfaces given in a Monge-patch parametrization $(x,y,h(x,y))$, where $h$ is harmonic), and thus this approach allows for the simple addition of helical motifs of arbitrary geometry and topology. However, the resulting surfaces are not exactly minimal. This is exceptionally apparent in the immediate vicinity of the helical structures, i.e. within a distance of a few pitch lengths from the core of each motif, as can be observed in Fig. \ref{fig:DeviationMinimalityGraph2helicoid}.

In the biologically and physically relevant regimes,  helical motifs are commonly separated by a distance comparable to the pitch, rendering the small slope approximation irrelevant for these structures \cite{BussiPNAS2019,TerasakiCell2013}. With the purpose of bridging this gap, in this work, we provide a recipe for constructing exact minimal surfaces with an arbitrary distribution of helical motifs. We first show that any harmonic graph (approximately minimal surface) can be deformed into an exact minimal surface through an explicit, yet non-local operation. We moreover show that every minimal surface could be obtained through this recipe. 
We conclude by surveying key examples: a pair of motifs of the same and opposite pitch, and an 
infinite chain of identical motifs  distributed along a line with similar or alternating handedness.

\section{Enneper immersions of minimal surfaces}
\label{Sect::EnneperImmersions}

\begin{figure}[t]
    \centering
    {\includegraphics[width=\linewidth]{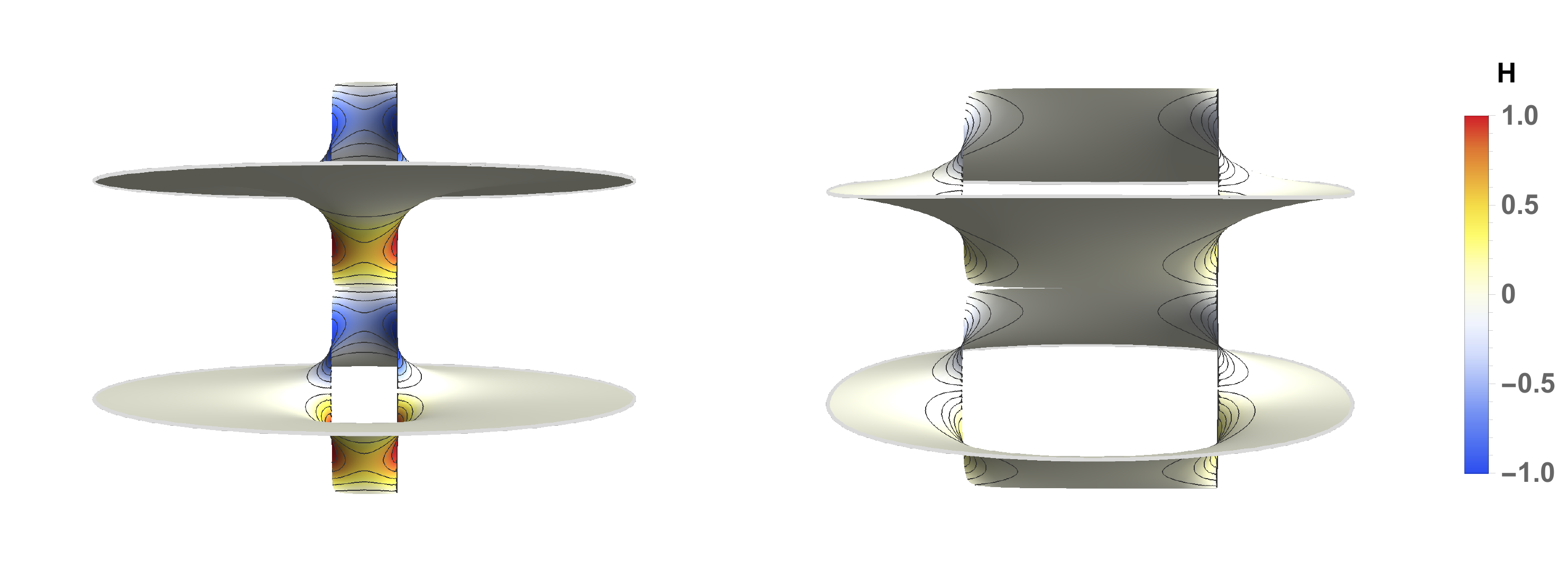}}
    \caption{Deviation from minimality in the small slope approximation as a function of $\frac{p}{R}$. The figure depicts two layers obtained from the graph of the harmonic function $h(z)=p\arg(z-\frac{R}{2})-p\arg(z+\frac{R}{2})$ colored by its mean curvature $H$. Notice that the deviation from minimality is most apparent for the closely-separated motifs and near the helical core. The size of the region where $H$ cannot be neglected increases with $\frac{p}{R}$. \textbf{(Left)} Plot for $p=0.5$ and $R=1.0$. \textbf{(Right)} Plot for $p=0.5$ and $R=4.0$.}
    \label{fig:DeviationMinimalityGraph2helicoid}
\end{figure}

We begin by presenting a somewhat underused representation of minimal surfaces termed \emph{Enneper immersions} due to Andrade \cite{AndradeJAM1998}. In essence, this technique allows us to ``fix'' an approximate minimal surface given as a graph over the plane such that it becomes an exact minimal surface by only exploiting lateral displacements and not varying the height data. Consider a surface given by a harmonic function {$h:\Omega\to\mathbb{R}$ over a domain $\Omega$} in the complex plane
\begin{equation}
\mathbf{r}(z) = (z,h(z)),\qquad z=x+\rmi y\in\Omega \subset \mathbb{C},\qquad \pd_{\bar{z}}\pd_{z}h=0,
\label{eq:harmonic-graph}
\end{equation}
{where we use $\bar{z}=x-\rmi y$, $\pd_{z}\equiv \tfrac{\pd}{\pd z}=\frac{1}{2}(\partial_x-\rmi\partial_y)$, and $\pd_{\bar{z}}\equiv \tfrac{\pd}{\pd \bar{z}}=\frac{1}{2}(\partial_x+\rmi\partial_y)$. For convenience, we will from now on also represent surfaces through complex (non necessarily analytic) functions since in the subsequent calculations we will need the derivatives of $h$ with respect to $z$ and $\bar{z}$ which results in complex functions that are not real valued.} 

The mean curvature $H$ of such a surface reads
\begin{equation}\label{Eq::MeanCurvGraphs}
H=\frac{(1+h_x^2)h_{yy}-2h_xh_yh_{xy}+(1+h_y^2)h_{xx}}{2(1+h_x^2+h_y^2)^{\frac{3}{2}}}=-2\frac{h_z^2h_{\bar{z}\bar{z}}-(1+2h_zh_{\bar{z}})h_{z\bar{z}}+h_{\bar{z}}^2h_{zz}}{(1+4h_zh_{\bar{z}})^{\frac{3}{2}}}.
\end{equation}
As $H=\frac{1}{2}\nabla^2h+\mathcal{O}(\Vert\nabla h\Vert^{2})=\mathcal{O}(\Vert\nabla h\Vert^{2})$, such surfaces are considered minimal within the small slope approximation (where terms of order $\mathcal{O}(\Vert\nabla h\Vert^{2})$ are omitted). A great advantage of this method is that due to the additivity of harmonic functions one can simply ``add'' helical motifs. The surface 
\begin{equation}\label{eq::SingleHelicoid}
\mathbf{r}(z) = (z,p_0\arg(z-z_0)),    
\end{equation}
yields a helicoid of pitch $p_{0}$ centered around $z_{0}$ which is exactly minimal. One can explicitly construct the surface which includes $N$ helical motifs located at the points $\{z_{k}\}_{k=1}^N$ and of corresponding pitches $\{p_{k}\}_{k=1}^N$ through
\[
\mathbf{r}(z) = (z,\mathop{\sum}^{N}_{k=1}p_k\arg(z-z_k)).
\]
The mean curvature of this surface, however, does not vanish identically. (The helicoids and planes are the only minimal surfaces that are also the graph of a harmonic function.) In particular, considering the case of a dipole of pitches $+p$ and $-p$ separated a distance $R$ from each other, say with axes located at $z_1=\frac{R}{2}$ and $z_2=-\frac{R}{2}$, the mean curvature reads
\[
H = -\frac{p^3R^3y}{C(x,y)^{\frac{1}{2}}[C(x,y)+p^2R^2]^{\frac{3}{2}}},\quad \text{where}\quad
\,C(x,y)=[(x-\frac{R}{2})^2+y^2][(x+\frac{R}{2})^2+y^2].
\]
For example, as $x=R/2$ and $y\to 0$ then $H\to -8/R$, or expressed in dimensionless terms $|H/\sqrt{-K}|=2p/R$. We thus see that the deviation from minimality becomes significant whenever the inter-motif separation becomes comparable to their pitch (which is the case in both \cite{TerasakiCell2013} and \cite{BussiPNAS2019}). To amend this we turn to construct a specific type of an Enneper immersion. We begin by stating the general result due to Andrade \cite{AndradeJAM1998}:  

Let  $h:\Omega\subset\mathbb{C}\to\mathbb{R}$ be a harmonic function and  $L,P:\Omega\subset\mathbb{C}\to\mathbb{C}$ be holomorphic. We may define  
\begin{equation}\label{Eq::EnneperGraph}
    X(z)=(L(z)-\overline{P(z)},h(z)),\,z\in\Omega.
\end{equation}
The mapping $X$ is a conformal minimal surface when $L$ and $P$ satisfy
\begin{equation}\label{Eq::CondForEnneperBeMinimal}
    L'(z)P'(z)=\left(\frac{\partial h}{\partial z}\right)^2\quad \mbox{ and }\quad \vert L'(z)\vert+\vert P'(z)\vert\not=0.
\end{equation}
The first condition guarantees that $X$ is conformal, i.e., {the coefficients $a_{11}=\langle \partial_xX,\partial_xX\rangle$, $a_{12}=\langle \partial_xX,\partial_yX\rangle$, and $a_{22}=\langle \partial_yX,\partial_yX\rangle$ of the metric satisfy} $a_{11}=a_{22}$ and $a_{12}=0$, while the second implies that the surface is regular, i.e., $\det (a_{ij})\not=0$. Conversely, every minimal surface can be written as the Enneper graph of a harmonic function. The reader is referred to \cite{AndradeJAM1998} for proof of the general case. Below we address a particular case where $L(z)=z$. 

\subsection{Obtaining exact minimal surfaces from harmonic approximations}

Consider an approximately minimal surface in the form of a harmonic graph \eqref{eq:harmonic-graph}, then the surface 
\begin{equation}\label{eq::EnneperDeformation}
\mathbf{r}(z)=(z-\overline{P(z)},h(z)), \quad \text{where} \quad P(z)=\int\left(\partial_{z} h\right)^2 dz,    
\end{equation}
is an exact minimal surface. Conversely, every minimal surface can be locally parametrized as in equation \eqref{eq::EnneperDeformation}. Moreover, the complex variable $z$ provides a conformal parametrization of the surface $\mathbf{r}$, and the area element $\rmd A$ and Gaussian curvature $K$ read:
\begin{equation}\label{Eq::AreaEnneperGraph}
    \rmd A = (1+\vert\partial_z h\vert^2)^2\rmd x\rmd y = (1+\tfrac{1}{4}\Vert\nabla h\Vert^2)^2\rmd x\rmd y
\end{equation}
and 
\begin{equation}\label{Eq::GaussCurvEnneperGraph}
    K=-  4\frac{\vert \partial_{zz} h\vert^2}{(1+\vert \partial_zh\vert^{2})^{4}}=\frac{h_{xx}h_{yy}-h_{xy}^2}{[1+\frac{1}{4}(h_x^2+h_y^2)]^{4}}.
\end{equation}

It is important to state that the $x,y$ coordinates above are not Cartesian coordinates on the plane, i.e., they do not provide a Monge patch for the surface $\mathbf{r}$, rather they are the image of these coordinates under the transformation $z\mapsto z-\overline{P(z)}$ \footnote{Distinguishing the coordinate systems is exceptionally important in view of the functional similarity between the above result and the expression for the Gaussian curvature of an arbitrary surface given as a graph: 
\[K_{(x,y,h(x,y))}=(h_{xx}h_{yy}-h_{xy}^2)/(1+h_x^2+h_y^2)^2\].}.

We next provide a rigorous proof of these claims. 

\subsubsection{Proof of equations \eqref{eq::EnneperDeformation}, \eqref{Eq::AreaEnneperGraph} and \eqref{Eq::GaussCurvEnneperGraph}}

Any immersion $\mathbf{r}=(L-\overline{P},h)$ can be rewritten as {$\mathbf{r}=\re\int(L'-P',-\rmi(L'+P'),2h_z)\rmd z,$} where the third coordinate was obtained by taking into account that for any real function $h$ it is valid $\rmd h=h_z\rmd z+h_{\bar{z}}\rmd\bar{z}=h_z\rmd z+\overline{h_{z}\rmd z}$, from which follows that $h=\int\rmd h=2\re(\int h_z\rmd z)$. Now, using the known results for holomorphic representations of minimal surfaces, see Appendix \ref{AppHolomRepMinSurf} for details, the necessary and sufficient conditions for $\mathbf{r}=\re(\int\phi_1\rmd z,\int\phi_2\rmd z,\int\phi_3\rmd z)$ to be a minimal immersion are that $\phi_{1},\phi_{2}$ and $\phi_{3}$ are holomorphic,  $\phi_1^2+\phi_2^2+\phi_3^2=0$, and $\vert\phi_1\vert^2+\vert\phi_2\vert^2+\vert\phi_3\vert^2\not=0$, whose validity for Enneper immersions follows from \eqref{Eq::CondForEnneperBeMinimal}. In addition, we obtain that the area and Gaussian curvature of $\mathbf{r}=(z-\bar{P},h)$ are given by equations \eqref{Eq::AreaEnneperGraph} and \eqref{Eq::GaussCurvEnneperGraph}.

For the converse, i.e., showing that any minimal surface is an Enneper immersion, Andrade uses the Uniformization Theorem to pass to the universal covering of $\Sigma$ and then build $L$, $P$, and $h$ \cite{AndradeJAM1998}. Here, we shall provide a more elementary proof to show that every minimal graph $\mathbf{r}(\zeta)=(\zeta,h(\zeta))$ can be written as an Enneper graph $\mathbf{r}(z)=(z-\overline{P(z)},h(z))$ after an appropriate change of coordinates $\zeta\mapsto z(\zeta)$. To find this change of coordinates we proceed as follows. First, note we may assume that neither $h_{\zeta}=0$ nor $h_{\bar{\zeta}}=0$. Otherwise, $h(\zeta)$ would be (anti-)holomorphic,  but no non-constant real function can be (anti-)holomorphic. Now, assuming we are able to reparametrize $\mathbf{r}(\zeta)=(\zeta,h(\zeta))$ as $\mathbf{r}(z)=\mathbf{r}(\zeta(z))=(z-\overline{P(z)},h(z))$ and using  $\zeta=z-\bar{P}$, gives
\[
P_z = (h_z)^2=(h_{\zeta}\zeta_z+h_{\bar{\zeta}}\bar{\zeta}_z)^2=h_{\zeta}^2-2P_zh_{\zeta}h_{\bar{\zeta}}+h_{\bar{\zeta}}^2P_z^2,
\]
which allows us to write
\[
P_z = \frac{1+2h_{\zeta}h_{\bar{\zeta}}\pm\sqrt{1+4h_{\zeta}h_{\bar{\zeta}}}}{2(h_{\bar{\zeta}})^2}\mbox{ and }P_{\bar{z}}=0.
\]

In order to find $P$ for a given minimal graph $\mathbf{r}(\zeta)=(\zeta,h(\zeta))$, we must be able to write a differential equation for $P$ in terms of $\zeta$ and $\bar{\zeta}$. To do that, we use  $P_z=P_{\zeta}\zeta_z+P_{\bar{\zeta}}\bar{\zeta}_z=P_{\zeta}-P_zP_{\bar{\zeta}}$ and $0=P_{\bar{z}}=P_{\zeta}\zeta_{\bar{z}}+P_{\bar{\zeta}}\bar{\zeta}_{\bar{z}}=P_{\bar{\zeta}}-\bar{P}_{\bar{z}}P_{\zeta}$ to find the system of differential equations
\begin{equation}\label{ODE::P_zetaP_zetaBar}
P_{\zeta}=\frac{B}{1-\vert B\vert^2}\mbox{ and }P_{\bar{\zeta}}=\frac{\vert B\vert^2}{1-\vert B\vert^2};\,B(\zeta):=\displaystyle\frac{1+2h_{\zeta}h_{\bar{\zeta}}+\sqrt{1+4h_{\zeta}h_{\bar{\zeta}}}}{2(h_{\bar{\zeta}})^2}.
\end{equation}
Notice, $B$ is the ``plus"  solution for $P_z$ and, consequently, the denominator of the ODE's in equation \eqref{ODE::P_zetaP_zetaBar} is
\[
1-\vert B\vert^2=-\frac{1+4h_{\zeta}h_{\bar{\zeta}}+(h_{\zeta}h_{\bar{\zeta}})^2+(1+2h_{\zeta}h_{\bar{\zeta}})\sqrt{1+4h_{\zeta}h_{\bar{\zeta}}}}{(h_{\zeta}h_{\bar{\zeta}})^2}<0,
\]
where we used that $h_{\zeta}h_{\bar{\zeta}}=\frac{1}{4}\Vert \nabla h\Vert^2$, which implies that the numerator of $\vert B\vert^2-1$ is a sum of positive quantities.

In short, the equations in \eqref{ODE::P_zetaP_zetaBar} provide  necessary conditions for the existence of the coordinate change $\zeta\mapsto\zeta(z)$. It remains to show they are also sufficient. From now on, let $\mathbf{r}(\zeta)=(\zeta,h(\zeta))$ be a minimal graph. Thus, it makes sense to write the system of ODE's in equation \eqref{ODE::P_zetaP_zetaBar}, whose solvability condition is $P_{\bar{\zeta}\zeta}=P_{\zeta\bar{\zeta}}$. Using the definition of $B$, we have
\[
P_{\zeta\bar{\zeta}}-P_{\bar{\zeta}\zeta}=\frac{B_{\bar{\zeta}}+B^2\bar{B}_{\bar{\zeta}}-B\bar{B}_{\zeta}-\bar{B}B_{\zeta}}{(1-\vert B\vert^2)^2}=\frac{(1+4h_{\zeta}h_{\bar{\zeta}})\,H}{2h_{\zeta}(h_{\bar{\zeta}})^2(1-\vert B\vert^2)}=0\Leftrightarrow H=0,
\]
where $H$ is the mean curvature of $\mathbf{r}(\zeta)=(\zeta,h(\zeta))$ as given in equation \eqref{Eq::MeanCurvGraphs}. In conclusion, we can integrate equation \eqref{ODE::P_zetaP_zetaBar} if and only if $\mathbf{r}(\zeta)=(\zeta,h(\zeta))$ is minimal. Thus, after integrating equation \eqref{ODE::P_zetaP_zetaBar} we can find $P(\zeta)$ and define $z$ as $z=\zeta+\overline{P(\zeta)}$. Since the Jacobian of the transformation $\zeta\mapsto z(\zeta)$ is
\[
\frac{\partial(z,\bar{z})}{\partial(\zeta,\bar{\zeta})}=
\left\vert
\begin{array}{cc}
z_{\zeta} & z_{\bar{\zeta}} \\
\overline{z_{\bar{\zeta}}} & \overline{z_{{\zeta}}} \\
\end{array}
\right\vert
=1+\bar{P}_{\zeta}+{P}_{\bar{\zeta}}+\vert\bar{P}_{\zeta}\vert^2-\vert{P}_{\bar{\zeta}}\vert^2 = 1+\bar{P}_{\zeta}=\frac{1}{1-\vert B\vert^2}<0,
\]
we can invert $\zeta\mapsto z(\zeta)$ and write $\zeta(z)=z-\overline{P(z)}$.

It remains to check that $P(z)=P(\zeta(z))$ is holomorphic, $h(z)=h(\zeta(z))$ is harmonic, and $P_z=(h_z)^2$. Notice, if $P(z)$ is holomorphic, then $P_z$ is also holomorphic and we conclude that $h_z=\sqrt{P_z}$ must be holomorphic, which implies $\Delta h=4h_{z\bar{z}}=0$. Thus, it suffices to prove $P(z)$ is holomorphic and $P_z=(h_z)^2$. From $\zeta=z-\bar{P}$, we find the system of equations
\[
P_{\zeta}=P_zz_{\zeta}+P_{\bar{z}}\bar{z}_{\zeta}=(1+\bar{P}_{\zeta})P_z+P_{\zeta}P_{\bar{z}}\mbox{ and }P_{\bar{\zeta}}=P_zz_{\bar{\zeta}}+P_{\bar{z}}\bar{z}_{\bar{\zeta}}=P_z\bar{P}_{\bar{\zeta}}+(1+P_{\bar{\zeta}})P_{\bar{z}}.
\]
Using the ODE's for $P_{\zeta}$ and $P_{\bar{\zeta}}$, we can solve this system and find that $P_z=B(\zeta(z))$ and $P_{\bar{z}}=0$, which in particular implies $P(z)$ must be holomorphic. In addition, using that $P_z=B$ and that $h_z=h_{\zeta}\zeta_z+h_{\bar{\zeta}}\bar{\zeta}_z=h_{\zeta}-\bar{P}_{\bar{z}}h_{\bar{\zeta}}$, we can straightforwardly deduce that $P_z=(h_z)^2$.

In conclusion, given a minimal surface $\Sigma$ parametrized as a graph $\mathbf{r}(\zeta)=(\zeta,h(\zeta))$, we can perform a change of coordinates $\zeta\mapsto z(\zeta)$ and reparametrize $\Sigma$ as an Enneper immersion in the particular form $\mathbf{r}(z)=(z-\overline{P(z)},h(z))$, i.e., we are able to see $\Sigma$ as the deformation of a harmonic graph.

\subsection{Basic example: helicoid}

Since a single helicoid is exactly minimal, one may expect that applying equation (\ref{eq::EnneperDeformation}) to equation (\ref{eq::SingleHelicoid}) would result in a trivial map, i.e., it would give $\mathbf{r}=(z,p_0\arg(z-z_0))$ back. However, an Enneper immersion is necessarily conformal, which is not the case for  $\mathbf{r}=(z,p_0\arg(z-z_0))$. The above intuition is, however, not entirely wrong; the Enneper graph of $h(z)=p_0\arg(z-z_0)$ yields the same helicoid but parametrized via a conformal immersion. 
 
Consider $h(z)=p_0\arg(z-z_0)=p_0\mbox{ Im}\ln(z-z_0)$. The remaining function $P$ composing the Enneper data is given by 
\[
\partial_zh=\frac{p_0}{2\rmi(z-z_0)}\Rightarrow P(z)=\int(\partial_zh)^2=\frac{p_0^2}{4(z-z_0)}.
\]
Thus, the Enneper graph of a single helicoid is $\mathbf{r}=(z-\frac{p_0^2}{4(\zb-\zb_0)},p_0\arg(z-z_0))$. Setting $z~-~z_0=r\rme^{\rmi\theta}$, the surface may be written as 
\[
\vr=\left(r(1-\frac{p_0^2}{4r^{2}})\rme^{\rmi\theta}+z_{0},p_0\arg(z-z_0)\right),
\]
where the distance from the axis is measured by $\rho=r(1-p_0^2/4r^2)$. To be more precise, the parametrization is a double covering of the helicoid since $\rho\in(-\infty,0]$ for $2r\leq\vert p_0\vert$ and $\rho\in[0,\infty)$ for $2r\geq\vert p_0\vert$. Notice that the entire boundary of the disc $r = |p_{0}|/2$ is singular and can be associated with the axis of the helicoid. The large distortion associated with the mapping and the singularities is a consequence of the conformality of the parametrization.  

The two copies of the helicoid can be easily distinguished using the surface's unit normal, which can be expressed as $\mathbf{N}=(p_0^2+\rho^2)^{-\frac{1}{2}}(p_0\sin\theta,p_0\cos\theta,\rho)\in\mathbb{S}^2$. It is immediate to see that in the exterior (interior) of the disc of radius $\frac{|p_{0}|}{2}$ around $z_{0}$ the normal takes values on the north (south) hemisphere of $\mathbb{S}^2$. Finally, from equations (\ref{Eq::AreaEnneperGraph}) and (\ref{Eq::GaussCurvEnneperGraph}), the area of the Enneper graph of a helicoid on the domain corresponding to an annulus of radii $r_2>r_1\geq \frac{\vert p_0\vert}{2}$ and the Gaussian curvature are
\begin{equation}\label{eq::Area1Helicoid}
    \int_{r_1}^{r_2}\int_0^{2\pi}(1+\frac{p_0^2}{4r^2})^2r\rmd r\rmd\theta = \pi(r_2^2-r_1^2)+\pi p_0^2\ln\frac{r_2}{r_1}-\frac{\pi p_0^4}{16}\left(\frac{1}{r_2^2}-\frac{1}{r_1^2}\right)
\end{equation}
and
\begin{equation}\label{eq::GaussCurv1Helicoid}
    K = -\frac{p_0^2}{r^4(1+\frac{p_0^2}{4r^2})^4}=-\frac{p_0^2}{(p_0^2+\rho^2)^2}.
\end{equation}
These results will allow us to compare the surfaces obtained for interacting motifs with that of a single helicoid.

\subsection{Finite collections of helical motifs and the multipole expansion}

\label{Sect::MultipoleExpansion}

The construction of an Enneper graph is associated with a harmonic function, whose singularities may be interpreted as  charges in analogy to point charges in electrostatic. This analogy suggests that we can resort to a multipole expansion by considering helical motifs as point charges. While in the 2d electrostatic the potential associated with a point charge has the logarithmic singularity $V(r,\theta)=V(r)=\ln\frac{1}{r}$, in the context of helical motifs point charges should have an arg-singularity $h(r,\theta)=h(\theta)=\theta$. (Up to constants, $h=\theta$ is the only purely angular 2d harmonic function while $h=\ln r$ is the only purely radial 2d harmonic function.) Given an electric charge density $\mu(x,y)$ entirely contained in a disc $D_R=\{(x,y):\sqrt{x^2+y^2}\leq R\}$, the electrostatic potential outside the disc is harmonic $\Delta V=0$ in $\mathbb{R}^2\backslash D_R$, but inside $D_R$ we have $\Delta V=2\pi\mu$. (We are employing a distinct sign convention that will prove to be more useful when extending the approach to our context.) Outside the disc $D_R$, $V$ has the multipole expansion \cite{JoslinMP1983multipole}
\begin{equation}
    V=p\ln r+\sum_{k=1}^{\infty}\frac{1}{r^k}[a_k\cos(k\theta)+b_k\sin(k\theta)]=\mbox{Re}\left(p\ln z+\sum_{k=1}^{\infty}\frac{c_k}{z^k}\right),
\end{equation}
where $c_k=a_k+\rmi b_k\in\mathbb{C}$ and $p\in\mathbb{R}$. The coefficients $p$, $a_k$, and $b_k$ are written as functions of the charge density $\mu$ according to
\begin{equation}\label{eqCoefMultpExpElectroPotential}
    p=\int_{\mathbb{R}^2}\mu\,\rmd x\rmd y\mbox{ and }c_k=-\frac{1}{k}\int_{\mathbb{R}^2}\mu z^k\rmd x\rmd y.
\end{equation}
Now, considering the imaginary part, the multipole expansion associated with a "helicoidal" charge density $\mu$ is
\begin{equation}\label{EqMultipoleExpHelicCharge}
    h=p\,\theta+\sum_{k=1}^{\infty}\frac{1}{r^k}[b_k\cos(k\theta)-a_k\sin(k\theta)]=\mbox{Im}\left(p\ln z+\sum_{k=1}^{\infty}\frac{c_k}{z^k}\right),
\end{equation}
where the coefficients are the same as in equation (\ref{eqCoefMultpExpElectroPotential}). 

For a single helicoid of pitch $p_0$ located at $0\in \mathbb{C}$, we know that the corresponding harmonic function is $h(r,\theta)=p_0\,\theta=p_0\,\mbox{Im}\ln z$. Thus, a single helicoidal charge $p_0$ is, as expected, associated with the distribution $\mu=p_0\,\delta(z)$, where $\delta$ is the Dirac delta function. (This distribution leads to $c_k=0$.)  On the other hand, for a helicoidal charge $p_0$ located at $z_0\in \mathbb{C}$, we have
$$
    p_0\ln (z-z_0)=p_0\ln z+p_0\ln(1-\frac{z_0}{z})=p_0\ln z-\sum_{k=1}^{\infty}\frac{p_0}{k}\left(\frac{z_0}{z}\right)^k,
$$    
which implies    
\begin{equation}
    h(z)=p_0\,\theta+\sum_{k=1}^{\infty}\frac{p_0}{k}\left(\frac{r_0}{r}\right)^k\sin k(\theta-\theta_0).
\end{equation}
This expression is in agreement with equation (\ref{EqMultipoleExpHelicCharge}) for the distribution $\mu=p_0\delta(z-z_0)$. Indeed,
\begin{equation}
  p_0=\int p_0\delta(z-z_0),-\int\frac{p_0z^k}{k}\delta(z-z_0)=-\frac{p_0z_0^k}{k}\Rightarrow h(z)=\mbox{Im}\left[p_0\ln z-\sum_{k=1}^{\infty}\frac{p_0}{k}\left(\frac{z_0}{z}\right)^k\right]. 
\end{equation}

In general, for a set of helicoidal charges $\{p_j\}$ located at $\{z_j\}_{j=1}^N\subset\mathbb{C}$, the charge density is given by $\mu=\sum_{j=1}^Np_j\,\delta(z-z_j)$. Then, the coefficients in the multipole expansion are
\begin{equation}\label{eq::Coeff2dMultipoleExp}
    p = \int\sum_{j=1}^Np_j\,\delta(z-z_j)=\sum_{j=1}^Np_j \mbox{ and }
    c_k=-\frac{1}{k}\int\sum_{j=1}^N p_jz^k\,\delta(z-z_j)=-\frac{1}{k}\sum_{j=1}^Np_jz_j^k.
\end{equation}
Finally, the corresponding harmonic function is
\begin{equation}
    h(r,\theta) = \Big(\sum_{j=1}^Np_j\Big)\theta+\sum_{k\geq1}\sum_{j=1}^N\frac{p_j}{k}\left(\frac{r_j}{r}\right)^k\sin k(\theta-\theta_j).    
\end{equation}
Notice that this power series converges absolutely as long as $r>r_k$ for all $k$, i.e., for all $z$ outside the disc of radius $\max_k\{r_k\}$. 

This expansion shows that the far-field behavior of the fundamental layer of any finite collection of helical motifs asymptotes to that of a helicoid whose pitch is given by the sum of the pitches of the individual motifs. Several such intertwined simple helicoids may be required to fully cover all leaves of the structure. Note, however, that an infinite collection of helical motifs, as well as motifs that are not perpendicular to the parallel layers they pierce, are not described by this expansion.
 
\section{Concatenating helical motifs}

\begin{figure}[t]
    \centering
    {\includegraphics[width=0.4\linewidth]{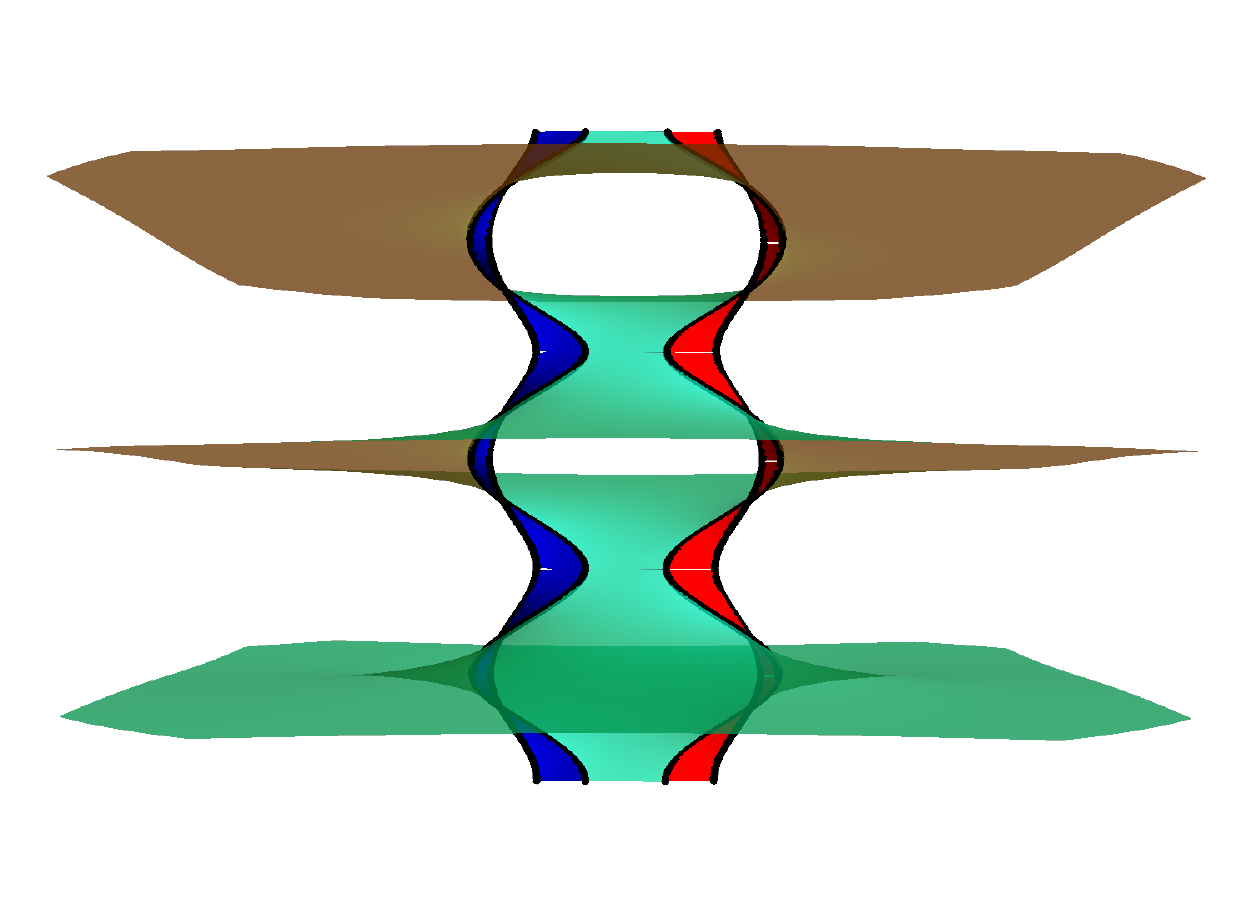}}
    {\includegraphics[width=0.35\linewidth]{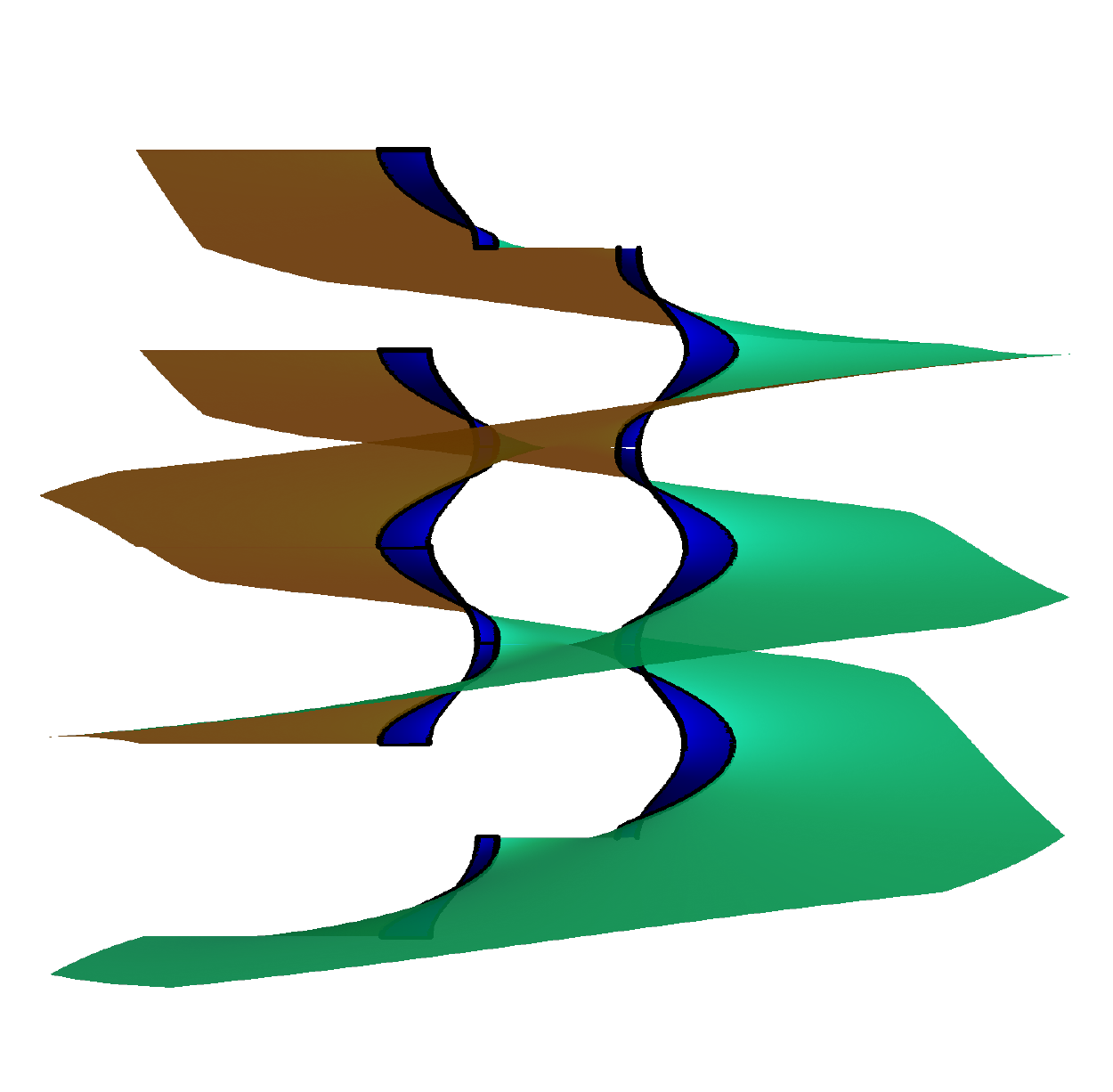}}
    \caption{Layers of a minimal 2-helicoid obtained by concatenating two oppositely handed or two equally handed helical motifs, left and right figures, respectively. (In the figures, $\vert p_{1}\vert=\vert p_{2}\vert=0.15$ and $R=1$.) Inner rims with distinct colors/shading emphasize the distinct handedness of the motifs.}
    \label{fig:2helicoids}
\end{figure}

Let us investigate in some detail the concatenation of two helical motifs, see Fig. \ref{fig:2helicoids}. Consider two helical motifs of pitches $p_1$ and $p_2$ located at $z_{1}$ and $z_{2}$, separated a distance $R$ from each other. The corresponding harmonic function is  given by $h(z)=p_1\arg(z-z_1)+p_2\arg(z-z_2)$, whose graph is only minimal up to leading order in $\Vert\nabla h\Vert$. Without loss of generality, we may assume $z_1=\frac{R}{2}$ and $z_2=-\frac{R}{2}$, where $R=\vert z_2-z_1\vert>0$. Since $\partial_zh=\frac{p_1}{2\rmi(z-z_1)}+\frac{p_2}{2\rmi(z-z_2)}$, integrating $(\partial_zh)^2$ we have
\begin{equation}
    P(z)=\frac{p_1^2}{4(z-\frac{R}{2})}+\frac{p_2^2}{4(z+\frac{R}{2})}+\frac{p_1p_2}{2R}\ln\left(\frac{z+\frac{R}{2}}{z-\frac{R}{2}}\right).
\end{equation}
The first two terms above are the sum of the functions $P(z;\{z_1,p_1\})$ and $P(z;\{z_2,p_2\})$ associated with helicoids around $z_1=\frac{R}{2}$ and $z_2=-\frac{R}{2}$. The surface obtained as the algebraic sum of the Enneper immersions of  $h(z;\{z_1,p_1\})=p_1\arg(z-\frac{R}{2})$ and $h(z;\{z_2,p_2\})=p_2\arg(z+\frac{R}{2})$ is not a minimal surface. The last term in $P(z)$, which can be written as $\frac{p_1p_2}{R}$ times the scale-independent function $\frac{1}{2}\ln[(\zeta+\frac{1}{2})/(\zeta-\frac{1}{2})]$, $\zeta=z/R$, thus introduces a pair-interaction between the two helicoids required to make the corresponding parametrization a conformal minimal immersion. {It is important to note that, in order to obtain a periodic layered surface, the pitches $p_1$ and $p_2$ should be commensurate.

In addition, the stereographic projection of the unit normal of a minimal 2-helicoidal parking garage is given by (see  Appendix \ref{AppHolomRepMinSurf} for the relation between the normal $\mathbf{N}$ and the Enneper data)
\begin{equation}
    g = -\left[\frac{p_1}{2\rmi(z-\frac{R}{2})}+\frac{p_2}{2\rmi(z+\frac{R}{2})}\right]^{-1} = -2\rmi R\frac{(\frac{z}{R})^2-\frac{1}{4}}{(p_1+p_2)\frac{z}{R}+\frac{1}{2}(p_1-p_2)}. \label{EqGaussMap2Helicoid}
\end{equation}
As the example of a single helicoid teaches us, we should pay attention in selecting the correct domain for a minimal 2-helicoid. The analysis of $g$ will help us identify the proper domain for the Enneper graph of a parking garage and we shall consider as the domain of the Enneper graph  the points $z\in\mathbb{C}$ where the unit normal $\mathbf{N}$ takes values on the North hemisphere: $\Omega_N=\{z\in\mathbb{C}:\vert g(z)\vert\geq 1\}$. The curves $\vert g\vert=1$, i.e., the boundary of $\Omega_N$,  (Fig. \ref{fig:LevelSets2-HelGaussMap}) will be associated with the helices of each helical motif, as explained in the following subsections.

\subsection{Concatenating two identical helical motifs}

\begin{figure}
    \centering
{\includegraphics[width=0.8\linewidth]{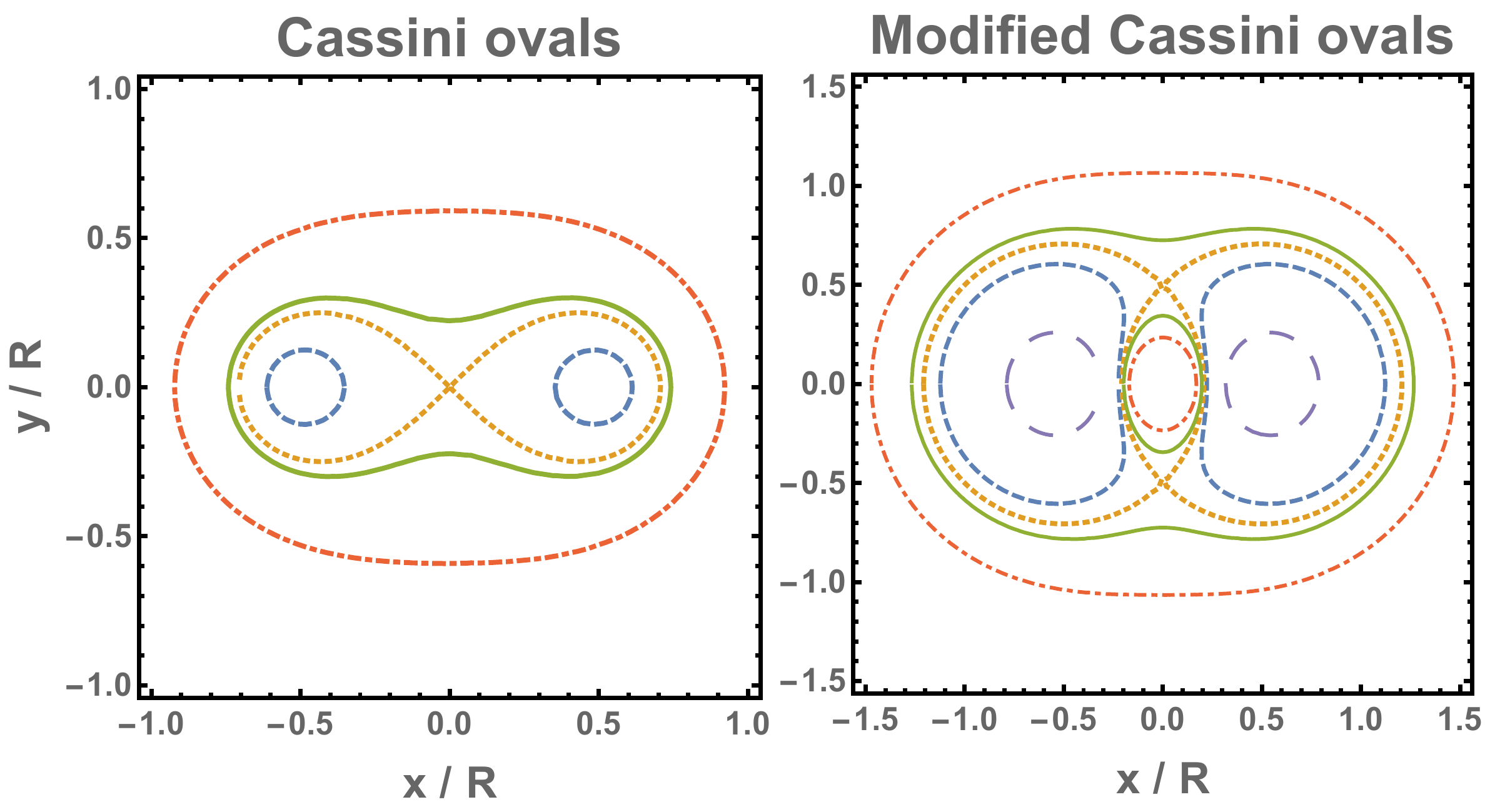}}
    \caption{Level curves of the mapping $z\mapsto\vert g(z)\vert$ associated with two helical motifs: \textbf{(Left)} For a minimal dipole, the curves are Cassini ovals with equation $\rho^4-\frac{1}{2}\rho^2\cos(2\theta)+\frac{1}{16}=\frac{\tilde{p}^2}{4}$, where $\tilde{p}=\vert g\vert\frac{p}{R}$ and $\rho=\frac{r}{R}$. Qualitatively, there are 4 types of shapes. (i) $\tilde{p}^2<\frac{1}{4}$: two disconnected and symmetric closed curves (dashed blue line). (ii) $\tilde{p}^2=\frac{1}{4}$: a  Lemniscate of Bernoulli, i.e., it has the shape of an 8 (dotted yellow line). (iii)  $\frac{1}{4}<\tilde{p}^2<1$: a peanut-shaped curve (thick green line). (iv) $\tilde{p}^2>1$: a closed convex curve (dotted-dashed red line); \textbf{(Right)} For a symmetric minimal non-dipole, the curves are modified Cassini ovals with equation $\rho^4-\frac{1}{2}\rho^2\cos(2\theta)+\frac{1}{16}=\tilde{p}^2\rho^2$, where $\tilde{p}=\vert g\vert\frac{p}{R}$ and $\rho=\frac{r}{R}$. Qualitatively there are 5 types of shapes. (i) $\tilde{p}^2<\frac{-1+\sqrt{5}}{2}$: two disconnected and symmetric closed convex curves (large-dashed violet line). (ii) $\frac{-1+\sqrt{5}}{2}<\tilde{p}^2<1$: two disconnected and symmetric closed non-convex curves (dashed blue line). (iii) $\tilde{p}^2=1$: the curves degenerate to two closed curves touching at two points (dotted yellow line). (iv) $1<\tilde{p}^2<\frac{1+\sqrt{5}}{2}$: two asymmetric pieces, a peanut-shaped curve and an inner closed convex curve (full green line). (v) $\tilde{p}^2>\frac{1+\sqrt{5}}{2}$: two disconnected and asymmetric  closed convex curves (dotted-dashed red line).}
    \label{fig:LevelSets2-HelGaussMap}
\end{figure}

For simplicity, we start by considering the symmetric problem of gluing two identical motifs such that $p_1=p_2=p$. For large distances (as we prove in Subsection \ref{Sect::EnneperImmersions}.\ref{Sect::MultipoleExpansion}) the resulting surface is well approximated by two intertwined and identical copies of helicoids of pitch $2p$ displaced a distance $p$ along the helicoids' axis with respect to each other, Fig. \ref{fig:2helicoids} (Right). The near field structure is, however, non-trivial.

The Gauss map in this case reads $g=-\rmi\frac{R}{p}[\frac{z}{R}-\frac{1}{4(z/R)}]$, from which follows that for every $w\in\mathbb{C}$, the equation $g(z)=w$ will have in general two distinct solutions. This means that the unit sphere is covered twice by the unit normal and, therefore, the total Gaussian curvature should be $-8\pi$. However, selecting as the domain of definition the points $z$ where the unit normal takes values on the North hemisphere, $\Omega_N$, the correct value of the total curvature of a minimal non-dipole is $-4\pi$, i.e., $-2\pi$ for each helical motif. The equator of the unit sphere is sent by the stereographic projection of the unit normal on the curve $z\mapsto\vert g(z)\vert=1$. Notice that
\begin{equation}
    \vert g\vert^2 = \frac{1}{(p/R)^2(r/R)^2}\left[\left(\frac{r}{R}\right)^4-\frac{1}{2}\left(\frac{r}{R}\right)^2\cos2\theta+\frac{1}{16}\right]=\frac{C(\frac{r}{R},\theta)}{(\frac{p}{R})^2(\frac{r}{R})^2},
\end{equation}
where $C(\rho,\theta)=\rho^4-\frac{1}{2}\rho^2\cos2\theta+\frac{1}{16}$ is a function whose level sets are the well known Cassini ovals (see Fig. \ref{fig:LevelSets2-HelGaussMap}, Left plot). It follows that the level curves of $\vert g\vert^2$ are modified Cassini ovals that can be analytically described in terms of fourth-degree polynomial curves whose shape is depicted in Fig. \ref{fig:LevelSets2-HelGaussMap} (Right plot). For small pitches, the region $\mathbb{C}-\Omega_N$ is made of two disconnected pieces around $z_1$ and $z_2$. However, when the pitches increase, the level sets display a topological transition at $p_c\equiv R$ and the complement of the domain $\Omega_N$ is no longer disconnected. The curves $\vert g\vert=1$, whose shape depends on $p/R$, parametrize the axes of the two helical motifs. By increasing $p/R$, we note that under the deformation in equation (\ref{eq::EnneperDeformation}) the two axes move toward each other and, consequently, the distance between them is smaller than the initial separation $R$ of the undistorted surface, see Fig. \ref{fig:effDistAxesAndInclinationAntiDip} (Center). For $p<p_c$, we may compute the points of closest approach of the two axes from the points of closest approach on the distinct connected components of $\vert g(z)\vert=1$. For the critical value $p=p_c$ the two axes finally intersect. In addition to getting closer along the line containing the motifs, the two axes also incline away from the line connecting  the motifs but in opposite senses, see Fig. \ref{fig:effDistAxesAndInclinationAntiDip} (Left). {Finally, the vertical separation between neighboring layers in the periodic structure remains unchanged during the deformation process as the height data provided by the harmonic function $h$ is preserved.}

To describe the geometry of the helical core in the symmetric case, we may compute the gradient and Hessian of $h(z)=p\arg(z+\frac{R}{2})+p\arg(z-\frac{R}{2})$. The gradient and Hessian of $h$ are
\begin{equation}
    \Vert\nabla h\Vert^2 =\frac{4\left(\frac{p}{R}\right)^2\left(\frac{r}{R}\right)^2}{\left[(\frac{r}{R})^2\cos2\theta-\frac{1}{4}\right]^2+(\frac{r}{R})^4\sin^22\theta}=4\left(\frac{p}{R}\right)^2\frac{\left(\frac{r}{R}\right)^2}{C(\frac{r}{R},\theta)}
\end{equation}
and
\begin{equation}
    \mbox{hess}\, h = -4\left(\frac{p}{R^2}\right)^2\frac{\left(\frac{x^2-y^2}{R^2}+\frac{1}{4}\right)^2+4\frac{x^2y^2}{R^4}}{\left[\left(\frac{x^2-y^2}{R^2}-\frac{1}{4}\right)^2+4\frac{x^2y^2}{R^4}\right]^2}=-4\left(\frac{p}{R^2}\right)^2\frac{C(\frac{r}{R},\theta+\frac{\pi}{2})}{C(\frac{r}{R},\theta)^2}.
\end{equation}
Therefore, from equation (\ref{Eq::GaussCurvEnneperGraph}), the Gaussian curvature of a minimal pair of identical motifs is
\begin{equation}\label{eq::GaussCurvAntiDipole}
    K =-\frac{4}{R^2}\left(\frac{p}{R}\right)^2\frac{C(\frac{r}{R},\theta+\frac{\pi}{2})C(\frac{r}{R},\theta)^2}{\left[C(\frac{r}{R},\theta)+\left(\frac{p}{R}\right)^2\left(\frac{r}{R}\right)^2\right]^4}\stackrel{r\gg R}{\approx}-\frac{4}{R^2}\left(\frac{p}{R}\right)^2\left(\frac{r}{R}\right)^{-4}=-\frac{4\, p^{2}}{r^{4}}.
\end{equation}
The infinitesimal area element of a minimal pair of identical motifs is
\begin{equation}
    \rmd A= \left[1+\frac{\left(\frac{p}{R}\right)^2\left(\frac{r}{R}\right)^2}{\left[\frac{x^2-y^2}{R^2}-\frac{1}{4}\right]^2+4\frac{x^2y^2}{R^4}}\right]^2 \rmd x\,\rmd y= R^2\left(\frac{r}{R}\right)\left[1+\frac{\left(\frac{p}{R}\right)^2\left(\frac{r}{R}\right)^2}{C(\frac{r}{R},\theta)}\right]^2 \rmd \left(\frac{r}{R}\right)\rmd \theta.
\end{equation}
It follows that the area of an annulus with large enough radii $r_{2}>r_{1}\gg R$ is given by
\begin{equation}
    \int_{r_1}^{r_2}\rmd A \approx \pi (r_2^2-r_1^2)+4\pi p^2\ln\frac{r_2}{r_1}+\mathcal{O}(\frac{1}{(r/R)^2}).
\end{equation}
Comparison with equation (\ref{eq::Area1Helicoid}) shows that, up to a small correction, the area of a minimal pair of identical motifs behaves for large distances as the area of a single helicoid of pitch $p_0=2p$. This is compatible with the expectation that a minimal pair of helical motifs can be well approximated by a helicoid of pitch $p_0=p_1+p_2$. (See the multipole expansion developed in Subsection \ref{Sect::EnneperImmersions}.\ref{Sect::MultipoleExpansion}.) Notice that the same is also valid for the behavior of the Gaussian curvature. Indeed, $K(r\gg R)\approx -4p^2/r^4$, which is precisely the asymptotic behavior of the Gaussian curvature of a helicoid of pitch $p_0=2p$, as confirmed by equation (\ref{eq::GaussCurv1Helicoid}).

\begin{figure}[t]
    \centering
{\includegraphics[width=\linewidth]{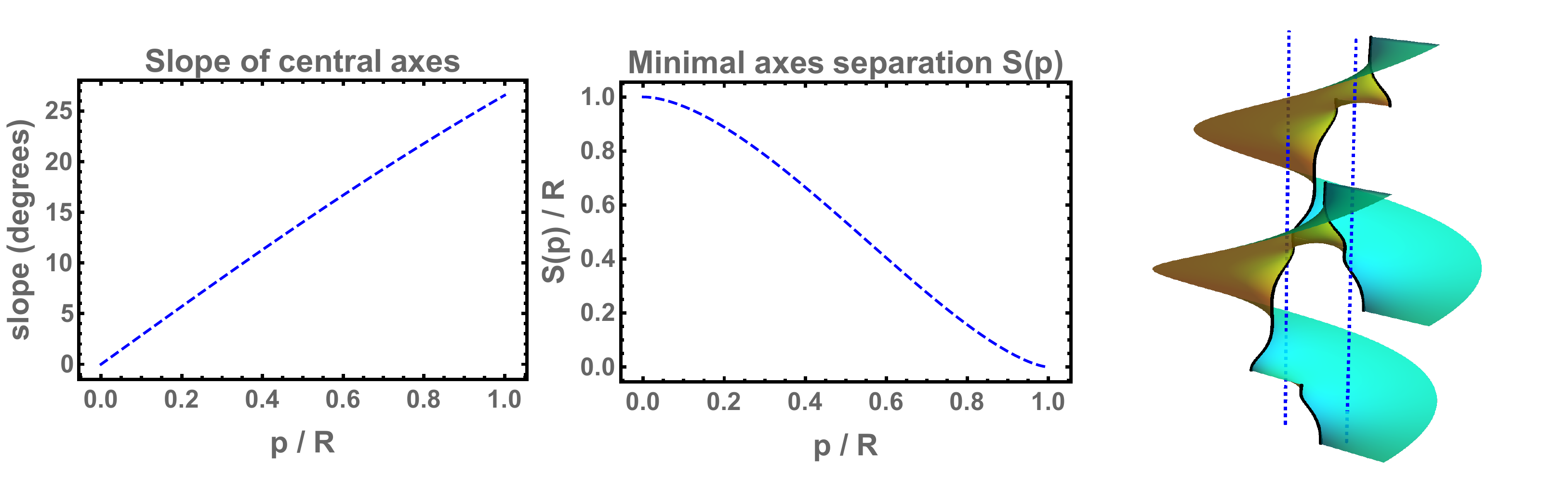}}
    \caption{\textbf{(Left)} The averaged inclination, with respect to the $z$-axis, of the two central axes of a minimal pair of equally handed helical motifs as a function of the pitch $p$ and initial separation $R$. The axes incline in opposite directions, toward and against the positive $y$-axis, which is orthogonal to the line containing the motifs. \textbf{(Center)} Minimal separation, $S(p)$, between the two axes of a minimal pair of equally handed helical motifs given by the closest points as a function of the pitch $p$ and initial separation $R$. \textbf{(Right)} Two layers of two equally handed motifs. The dashed blue lines represent the helical axes before the deformation. (In the picture, $p=0.5$ and $R=1$.)}
    \label{fig:effDistAxesAndInclinationAntiDip}
\end{figure}

\subsection{Minimal helicoidal dipole}

\begin{figure}[t]
    \centering
    {\includegraphics[width=\linewidth]{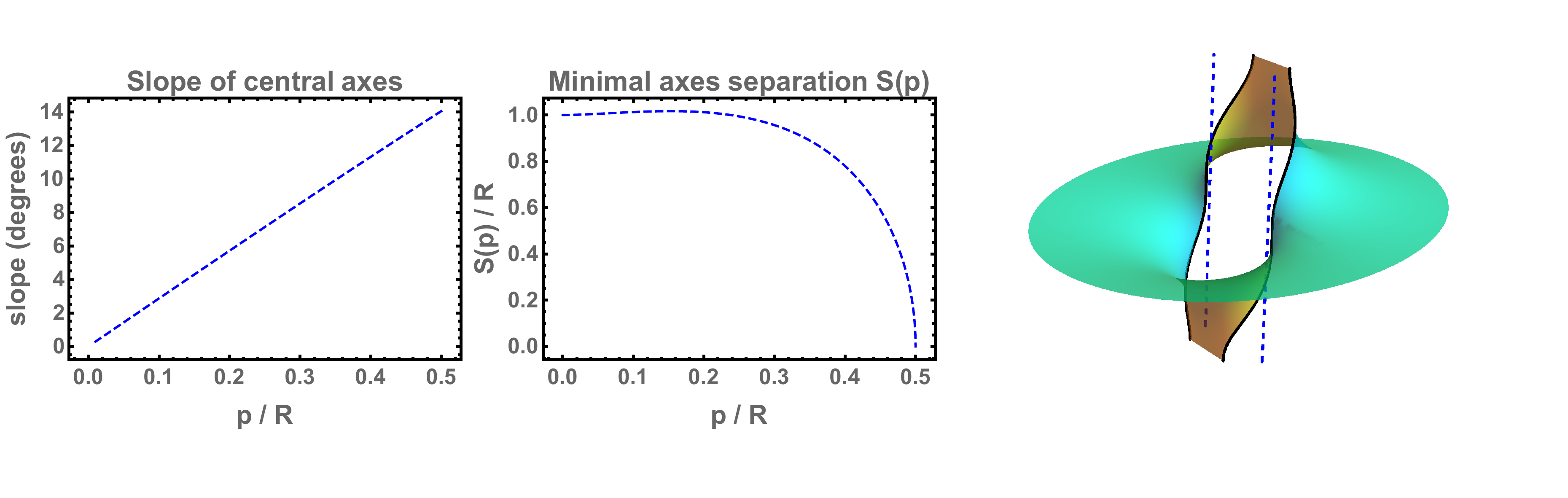}}
    \caption{\textbf{(Left)} The averaged inclination, toward the $y$-direction and with respect to the $z$-axis, of the two central axes of a minimal dipole as a function of the pitch $p$ and initial separation $R$. \textbf{(Center)} Minimal separation, $S(p)$, between the two axes of a minimal dipole given by the closest points as a function of the pitch $p$ and initial separation $R$. \textbf{(Right)} One layer of a minimal dipole. The dashed blue lines represent the helical axes before the deformation. (In the picture, $p=0.45$ and $R=1$.)}
    \label{fig:effDistAxesAndInclinationDip}
\end{figure}

Now, assume that the total pitch $p_0=p_1+p_2$ vanishes. Then, writing $p=p_1=-p_2$, the Gauss map becomes $g=-2\rmi\frac{R}{p}[(z/R)^2-1/4]$. This implies that the unit sphere is covered twice and, therefore, selecting as the domain of definition the points $z$ where the unit normal takes values on the North hemisphere, $\Omega_N$, the total Gaussian curvature of a minimal dipole is $-4\pi$, i.e., $-2\pi$ for each helical motif. In addition, noticing that
\begin{equation}
    \vert g\vert^2 = \frac{4R^2}{ p^2}C(\frac{r}{R},\theta),
\end{equation}
the curves defined by the level sets of $\vert g\vert^2$ are the well known Cassini ovals. Thus, the shape of the curve associated with the equator of $\mathbb{S}^2$, i.e., $z\mapsto\vert g(z)\vert=1$, is completely determined by the parameter ${\tilde{p}}=\vert p\vert/R$, as described in Fig. \ref{fig:LevelSets2-HelGaussMap} (Left plot). Here the curves $\vert g\vert=1$, whose shape depends on $p/R$, parametrize the axes of the two helical motifs. Under the deformation in equation (\ref{eq::EnneperDeformation}) the central part of two axes of a dipole pair do not converge toward each other upon increasing $p/R$, but their extremities do. For $p<p_c\equiv\frac{R}{2}$, the points of closest approach of the two axes can be computed from the points of closest approach on the distinct connected components of $\vert g(z)\vert=1$, see Fig. \ref{fig:effDistAxesAndInclinationDip} (Center plot). The level curves of $\vert g(z)\vert$ have a topological transition at $p_c=\frac{R}{2}$ and it follows that for $p>p_c$ the two axes of the Enneper graph of two oppositely handed helical motifs of pitch $p$ and $-p$ merge into a single smooth curve. In addition, the two axes effectively incline in the direction orthogonal to the line containing the motifs (both in the same sense), see Fig. \ref{fig:effDistAxesAndInclinationDip} (Left plot). {Finally, the vertical separation between neighboring layers in the periodic structure remains unchanged during the deformation process.}

As in the previous case, to describe the geometry of the helical core in the dipole case, we may compute the gradient and Hessian of $h(z)=p\arg(z+\frac{R}{2})-p\arg(z-\frac{R}{2})$. The gradient and Hessian of $h$ are
\begin{equation}
    \Vert\nabla h\Vert^2 = \left(\frac{p}{R}\right)^2 \frac{1}{\left[\left(\frac{r}{R}\right)^2-\frac{1}{4}\right]^2+\left(\frac{r}{R}\right)^2\sin^2\theta}=\left(\frac{p}{R}\right)^2 \frac{1}{C(\frac{r}{R},\theta)}
\end{equation}
and
\begin{equation}
    \mbox{hess}\, h = -\frac{4}{R^2} \left(\frac{p}{R}\right)^2 \frac{\left(\frac{r}{R}\right)^2}{\Big\{[(\frac{r}{R})^2-\frac{1}{4}]^2+\left(\frac{r}{R}\right)^2\sin^2\theta \Big\}^2}=-\frac{4}{R^2} \left(\frac{p}{R}\right)^2 \left(\frac{r}{R}\right)^2\frac{1}{C(\frac{r}{R},\theta)^2}.
\end{equation}
Therefore, from equation (\ref{Eq::GaussCurvEnneperGraph}), the Gaussian curvature of a minimal dipole is
\begin{equation}\label{eq::GaussCurvDipole}
    K =-\frac{4}{R^2} \left(\frac{p}{R}\right)^2\left(\frac{r}{R}\right)^2 \frac{C(\frac{r}{R},\theta)^2}{[C(\frac{r}{R},\theta)+\frac{1}{4}(\frac{p}{R})^2]^4}\stackrel{r\gg R}{\approx}-\frac{4}{R^2}\left(\frac{p}{R}\right)^2\left(\frac{r}{R}\right)^{-6}.
\end{equation}
Notice that $K$ decays to zero faster than the Gaussian curvature of a minimal non-dipole. On the other hand, the infinitesimal area of a minimal dipole is
\begin{equation}
    \rmd A=\left[1+\frac{1}{4C(\frac{r}{R},\theta)}\left(\frac{p}{R}\right)^2\right]^2 \rmd x\,\rmd y=R^2\left(\frac{r}{R}\right)\left[1+\frac{1}{4C(\frac{r}{R},\theta)}\left(\frac{p}{R}\right)^2\right]^2 \rmd\left(\frac{r}{R}\right)\rmd \theta.
\end{equation}
It follows that the area of an annulus with large enough radii $r_2>r_1\gg R$ is given by
\begin{equation}
    \int_{r_1}^{r_2}\rmd A \approx \pi (r_2^2-r_1^2)+\mathcal{O}(\frac{1}{(r/R)^2}).
\end{equation}
The non-trivial contribution is nothing but the area of a planar annulus.  This is compatible with the fact that for large distances, a minimal pair of opposite helical motifs is well approximated by a single helicoid of pitch $p_0=p-p=0$, i.e., a minimal dipole is approximately a plane at large distances. (See the multipole expansion developed in Subsection \ref{Sect::EnneperImmersions}.\ref{Sect::MultipoleExpansion}.)

\subsection{Gluing finitely many helical motifs}

For an arbitrary number of helicoids of pitches $p_1,\dots,p_N$ located at $z_1,\dots,z_N$, we may consider the harmonic function 
\begin{equation}\label{eq::hOfManyMotifs}
    h(z) = \sum_{k=1}^N p_k\im[\ln(z-z_k)]= \sum_{k=1}^N p_k\arg(z-z_k)\Rightarrow \frac{\partial h}{\partial z}=\sum_{k=1}^N\frac{p_k}{2\rmi}\frac{1}{z-z_k}.
\end{equation}
Now, since $P'=(\partial_zh)^2$, we have after integration
\begin{equation}
    P(z) = \frac{1}{4}\sum_k\frac{p_k^2}{z-z_k}-\frac{1}{2}\sum_{k<j}\frac{p_jp_k}{z_j-z_k}\ln\left(\frac{z-z_j}{z-z_k}\right)+\mbox{const. }.
    \label{eq:PofMany}
\end{equation}
Each term in the first sum above corresponds to the  function $P(z;\{z_k,p_k\})$ associated with a single helicoid of pitch $p_k$ located at $z_k$. The remaining terms can be seen as pair interactions between distinct helicoids required to assure that the corresponding parametrization is a conformal minimal immersion.

The stereographic projection of the layer normal $\mathbf{N}$ is given by
\begin{equation}
    g = -\left(\sum_{k=1}^N\frac{p_k}{2\rmi}\frac{1}{z-z_k}\right)^{-1}=-2\rmi\,\frac{z^N-a_1z^{N-1}+a_2z^{N-2}\dots+(-1)^Na_N}{d_1z^{N-1}-d_2z^{N-2}+\cdots+(-1)^{N-1}d_{N}}. 
\end{equation}
The coefficients $a_k$ and $d_k$ are computed from $\{p_j\}$ and $\{z_j\}$ as
\[
 a_k=s_k(z_1,\dots,z_N)\mbox{ and }d_k=\displaystyle\sum_{j=1}^N p_js_{k-1}(z_1,\dots,z_{j-1},z_{j+1},\dots,z_N),
\]
where $s_k$ denotes the $k$-th symmetric polynomial, $s_0(x_1,\dots,x_m)=1$ and $s_k(x_1,\dots,x_m)=\sum_{1\leq i_1<\dots<i_k\leq m}x_{i_1}\cdots x_{i_k}$. 

\subsection{Twist grain boundary - linear chain of infinitely many helical motifs}

Twist grain boundaries (TGB) are structures composed of an infinite collection of identical helical motifs evenly spaced along a straight line. These structures have been studied extensively in the context of smectic liquid crystals \cite{RennPRA1988}, where the individual helical motifs are termed screw dislocations after their crystalline cousins. Unlike individual motifs of finite charge which asymptote to parallel planes, TGB's asymptote to two different families of parallel planes on the two sides of the TGB line, Fig. \ref{fig:TGB} (Left). The two families of planes form a finite  angle with respect to one another, whose magnitude is determined by the ratio of inter-motif distance and its pitch. 

The distinct layers in a smectic structure are often modeled within the small slope approximation as level sets of the phase field  \cite{SantangeloPRL2006}
\[
\phi(x_1,x_2,x_3) = \lambda x_3-\frac{b}{2\pi}\mbox{Im}\ln\sin\left[\frac{\pi(x_1+\rmi x_2)}{\ell_d}\right].
\]
Equivalently, individual layers may be identified with the graph of the harmonic function 
\[
h(z=x+\rmi y) = \frac{b}{2\pi\lambda}\mbox{Im}\ln\sin\left(\frac{\pi z}{\ell_d}\right).
\] 
The resulting surface is minimal only to order $\mathcal{O}(\Vert\nabla h\Vert^2)$. Thus this approximation was primarily employed to address small angle variations in the surfaces normals, and TGB's comprised of helical motifs that are well separated compared with their pitch. An exact minimal  TGB surface allows estimating the elastic energy in cases where the inter-motif distance is comparable to the pitch. 

{Twist grain boundaries have an exact formulation provided by the family of Scherk doubly periodic minimal surfaces. In addition, it can be also shown that this surface can be seen as an infinite superposition of helicoid-like surfaces \cite{KamienAML2001}, though each term appearing in the superposition is not exact minimal. However, it is not clear how to generalize Scherk's construction to a family of helical motifs of different pitch values as well as to finitely many helical motifs, as done in the previous subsections. To illustrate the robustness of our method, here we construct a minimal TGB as an Enneper graph\footnote{{Our minimal TGB has a strong resemblance with the doubly periodic Scherk surface. Since they are similarly doubly periodic, they should coincide as a consequence of the topological characterization of Scherk minimal surfaces \cite{Lazard-HollyIM2001}.}} and, later, we also consider a minimal Untwisted Grain Boundary (UtGB), i.e., an infinite chain of evenly spaced helical motifs of alternating handedness, see Fig. \ref{fig:TGB} right.}

To achieve an exact formulation in the most transparent form, we start from considering the Enneper immersion of a linear stack of infinitely many helical motifs of pitch $p$,  evenly spaced a distance $\ell_d$ from each other: $ h(z) = \sum_{k=-\infty}^\infty p\,\im[\ln(z-k \ell_d)]$. Differentiating we obtain 
\begin{equation}
\frac{\partial h}{\partial z}=\sum_{k=-\infty}^\infty\frac{p}{2\rmi}\frac{1}{z-k \ell_d}=
\frac{p}{2\rmi}\left(\sum_{k=1}^\infty\frac{2z}{z^2-k^2 \ell^2_d}+\frac{1}{z}\right)=
\frac{p}{2\rmi}\frac{\pi}{\ell_d}\cot(\frac{\pi z}{\ell_d}),
\end{equation}
where we used the identity $\pi\cot(\pi z)=z^{-1}+2z\sum_{n\geq1}(z^2-n^2)^{-1}$. Squaring and integrating we obtain the correction term to produce the desired Enneper immersion
\[
P(z)=\int \left( \frac{\partial h}{\partial z}\right)^2 \rmd z=-\frac{p^2}{4}\int \frac{\pi^2}{\ell_d^2} \left(\csc^2(\frac{\pi z}{\ell_d})-1\right) \rmd z=\frac{p^2}{4}\frac{\pi}{\ell_d}\cot(\frac{\pi z}{\ell_d})+\frac{p^2}{4}\frac{\pi^2}{\ell_d^2} z.
\]
\begin{figure}[t]
    \centering
    {\includegraphics[width=0.35\linewidth]{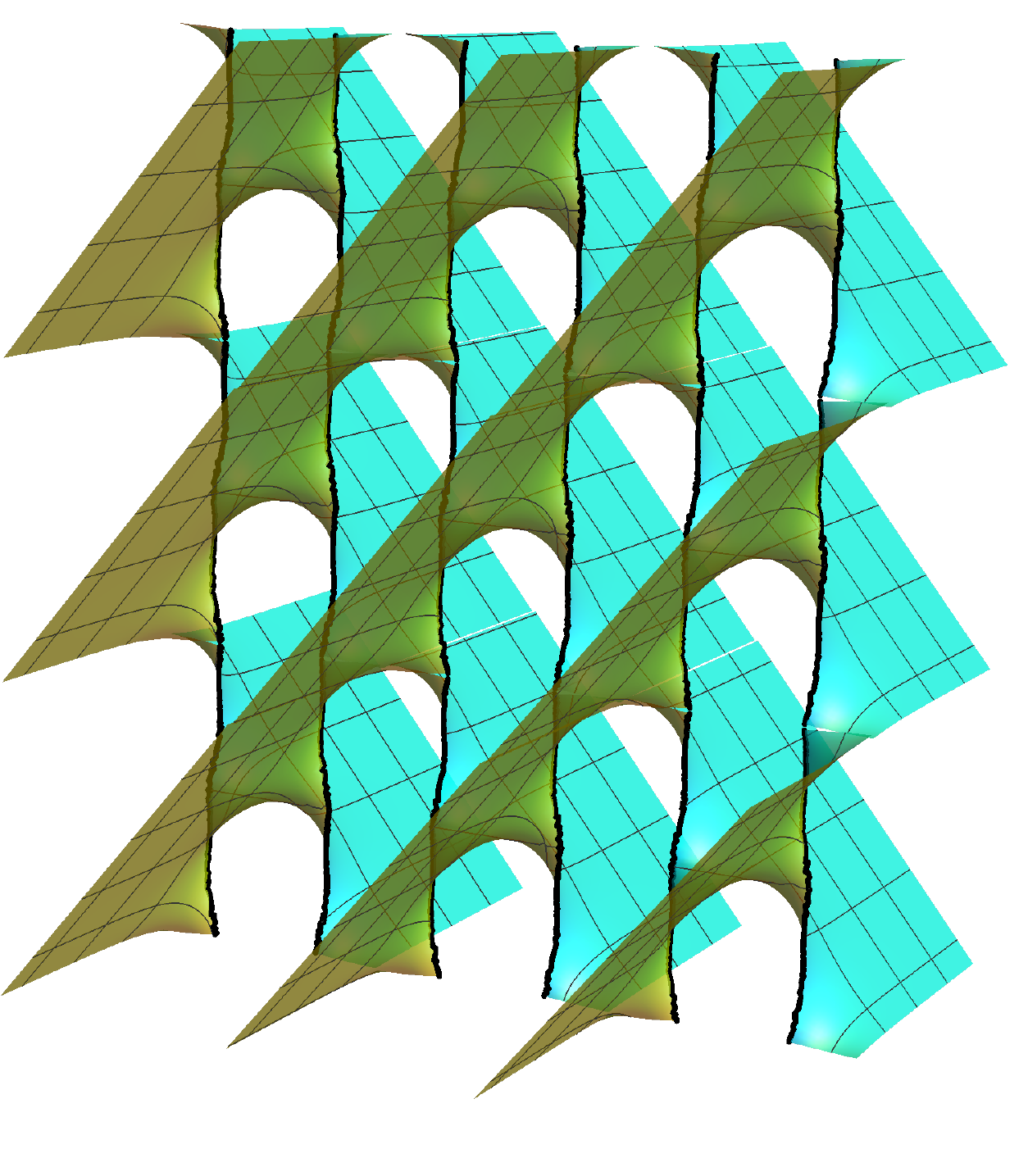}}\,\,\,{\includegraphics[width=0.4\linewidth]{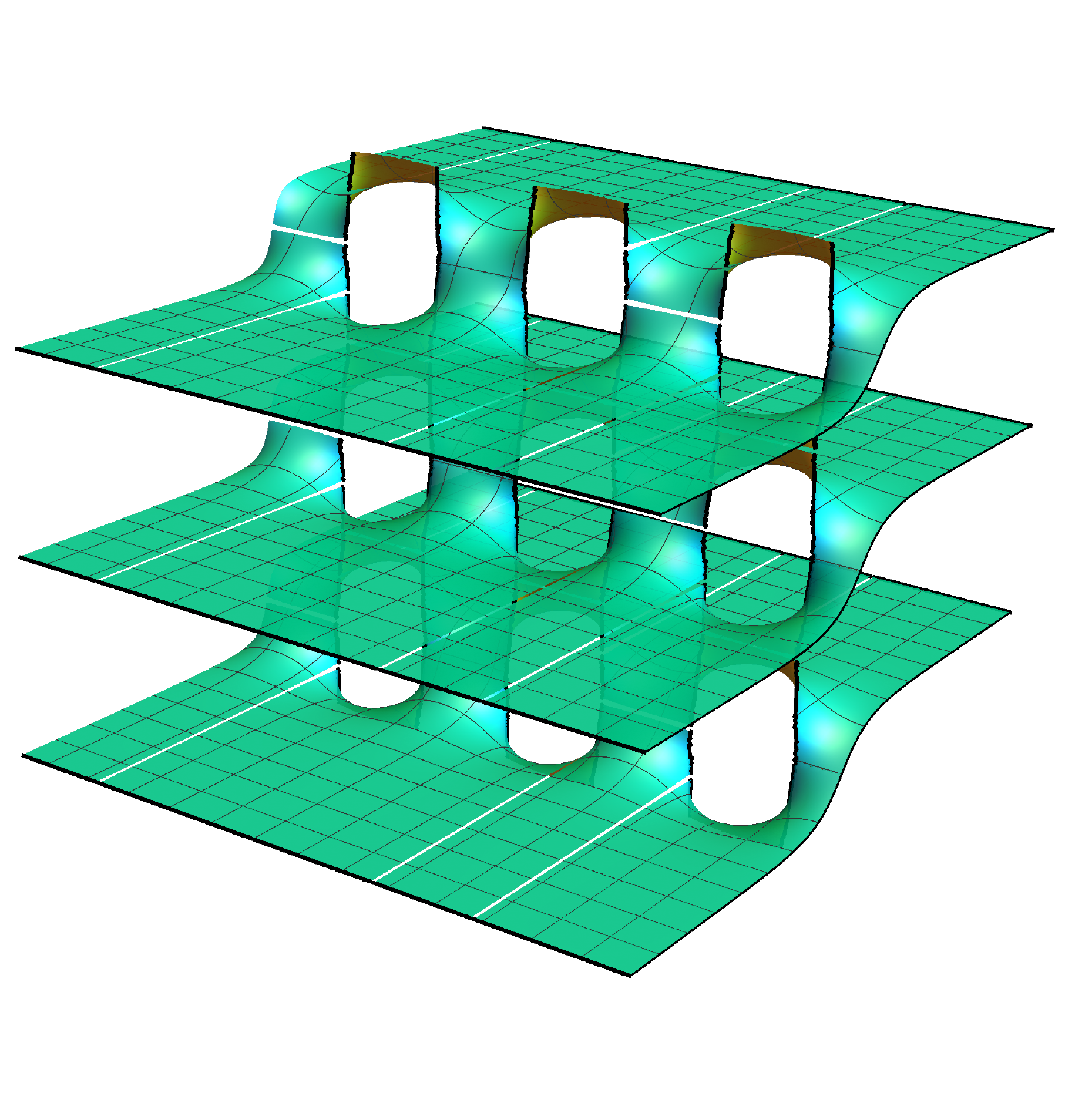}}
    \caption{\textbf{(Left)} Layers of a minimal Twist Grain Boundary (TGB) obtained by concatenating infinitely many helical motifs of same pitch $p$. \textbf{(Right)} Layers of a minimal Untwisted Grain Boundary (UtGB) obtained by concatenating infinitely many helicoidal dipoles of pitches $p$ and $-p$. (In the figure, $p=0.3$ and $\ell_d=1$.)}
    \label{fig:TGB}
\end{figure}
Alternatively, we can rewrite $P(z)$ as an infinite sum of {the functions $P(z;\{z_k,p_k\})$ associated with} helicoids of pitch $p$ located at $\{n\ell_d:n\in\mathbb{Z}\}$ plus a linear correction:
\begin{equation}
    P(z) = \sum_{n\in\mathbb{Z}}\frac{p^2}{4(z-n\ell_d)}+\frac{p^2\pi^2}{4\ell_d^2}z.
\end{equation}

The stereographic projection of the layer unit normal is
\begin{equation}
    g=-\frac{1}{\partial_zh}=\frac{2\ell_d}{p\pi\rmi}\tan\left(\frac{\pi z}{\ell_d}\right).
\end{equation}
As in the concatenation of finitely many helical motifs, here the level sets of $\vert g\vert$ help to find the proper domain to compute the Enneper graph of a minimal TGB (see Fig. \ref{fig:LevelSetsTGBGaussMap}, left). In addition, the image of the surface unit normal $\mathbf{N}$ over the fundamental domain $\Omega_0=\{(x,y)\in[0,2\ell_d]\times\mathbb{R}:\vert g(x+\rmi y)\vert\geq1\}$ covers the unit sphere exactly twice, which results in a total curvature of $\int_{\Omega_0}K\rmd A=-4\pi$: one copy comes from $\Omega_0\cap[0,\ell_d]\times\mathbb{R}$ and another from $\Omega_0\cap[\ell_d,2\ell_d]\times\mathbb{R}$. Since there are infinitely many helical motifs periodically distributed on a line, the total curvature of a minimal TBG is infinite. The Gaussian curvature $K$ is computed from the first and second derivatives of $h$. The second derivative is $\partial^2_zh=-\frac{p\pi^2}{2\rmi\ell_d^2}\csc^2(\frac{\pi z}{\ell_d})$, from which follows that
\begin{equation}
    K = -\frac{p^2\pi^4}{\ell_d^4}\frac{\vert\csc(\frac{\pi z}{\ell_d})\vert^4}{\left(1+\frac{p^2\pi^2}{4\ell_d^2}\vert\cot(\frac{\pi z}{\ell_d})\vert^2\right)^4}= -\frac{p^2\pi^4}{\ell_d^4}\frac{\vert\sin(\frac{\pi z}{\ell_d})\vert^4}{\left(\vert\sin(\frac{\pi z}{\ell_d})\vert^2+\frac{p^2\pi^2}{4\ell_d^2}\vert\cos(\frac{\pi z}{\ell_d})\vert^2\right)^4}.
\end{equation}

Finally, upon application of the deformation in equation (\ref{eq::EnneperDeformation}) leading to the Enneper graph of a minimal TGB, the effective distance between neighboring axes diminishes under the increase of $p/\ell_d$, as depicted in Fig. \ref{fig:effDistAxesTGB} (Left). It then follows that in order to obtain a minimal TGB with prescribed pitch and axes distance, the value of $\ell_d$ in the expression of $h(z)$ has to be appropriately tuned. The pitch remains unchanged during the deformation process.

\begin{figure}
    \centering
    {\includegraphics[width=0.8\linewidth]{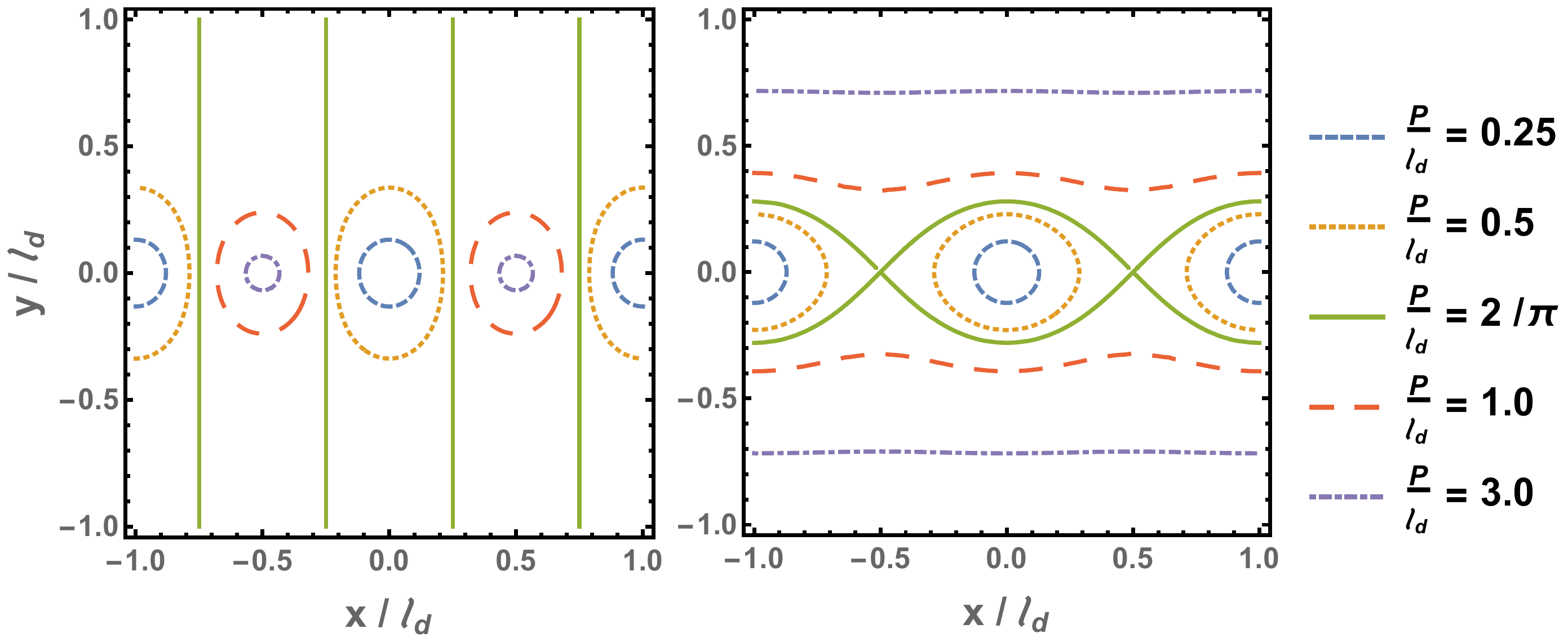}}
    \caption{Level curves of the Gauss map $z\mapsto\vert g(z)\vert$ given by the stereographic projection of the unit normal of minimal surfaces obtained by concatenating infinitely many helical motifs in a line. \textbf{(Left)} Level curves associated with a TGB, i.e., an infinite collection of identical motifs evenly spaced along a line. \textbf{(Right)} Level curves associated with UtGB, i.e., an infinite collection of helicoidal dipoles evenly spaced along a line. In both cases, the critical pitch is given by $p_c=\frac{2\ell_d}{\pi}$.}
    \label{fig:LevelSetsTGBGaussMap}
\end{figure}
\begin{figure}
    \centering
    {\includegraphics[width=0.44\linewidth]{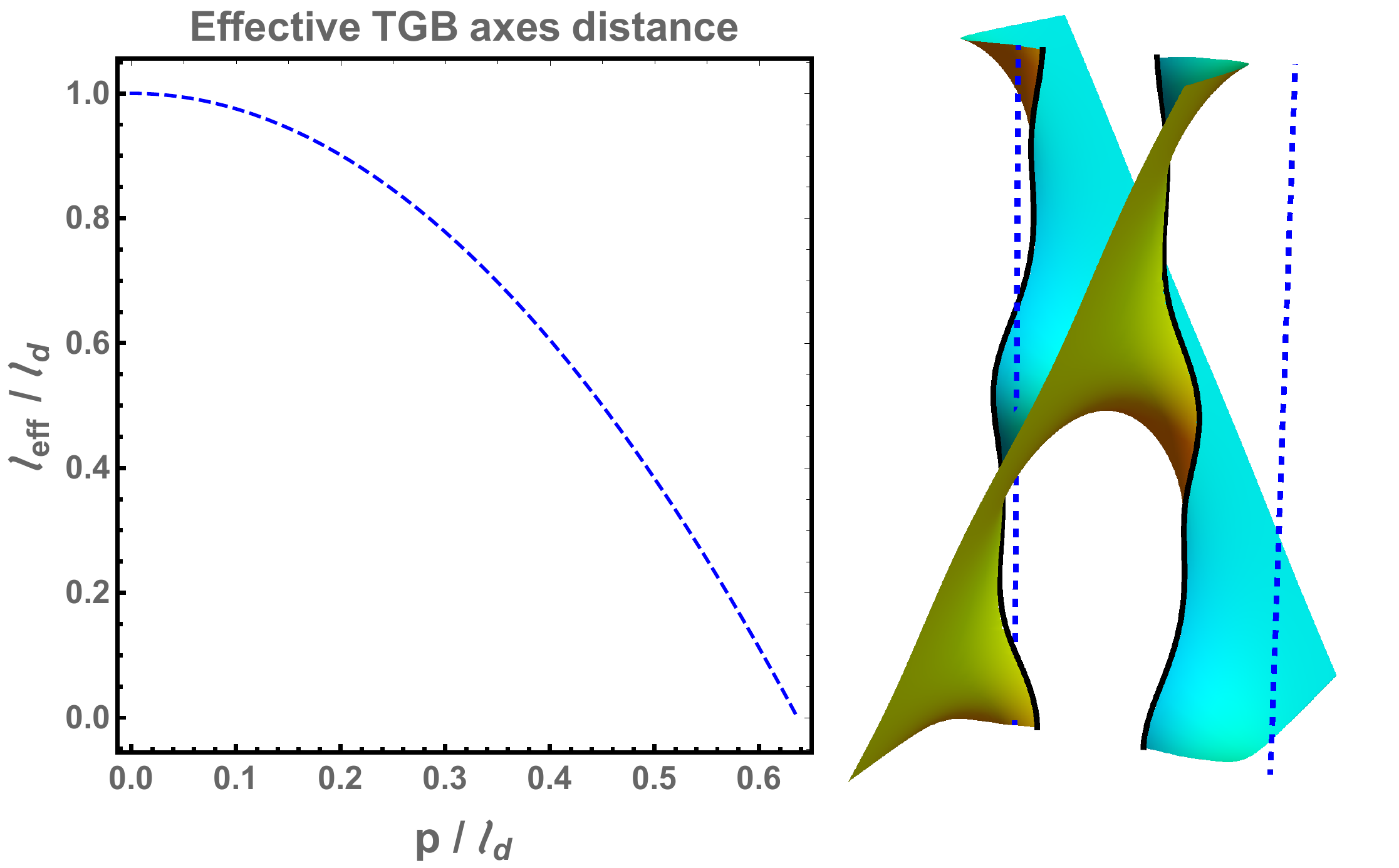}}
{\includegraphics[width=0.535\linewidth]{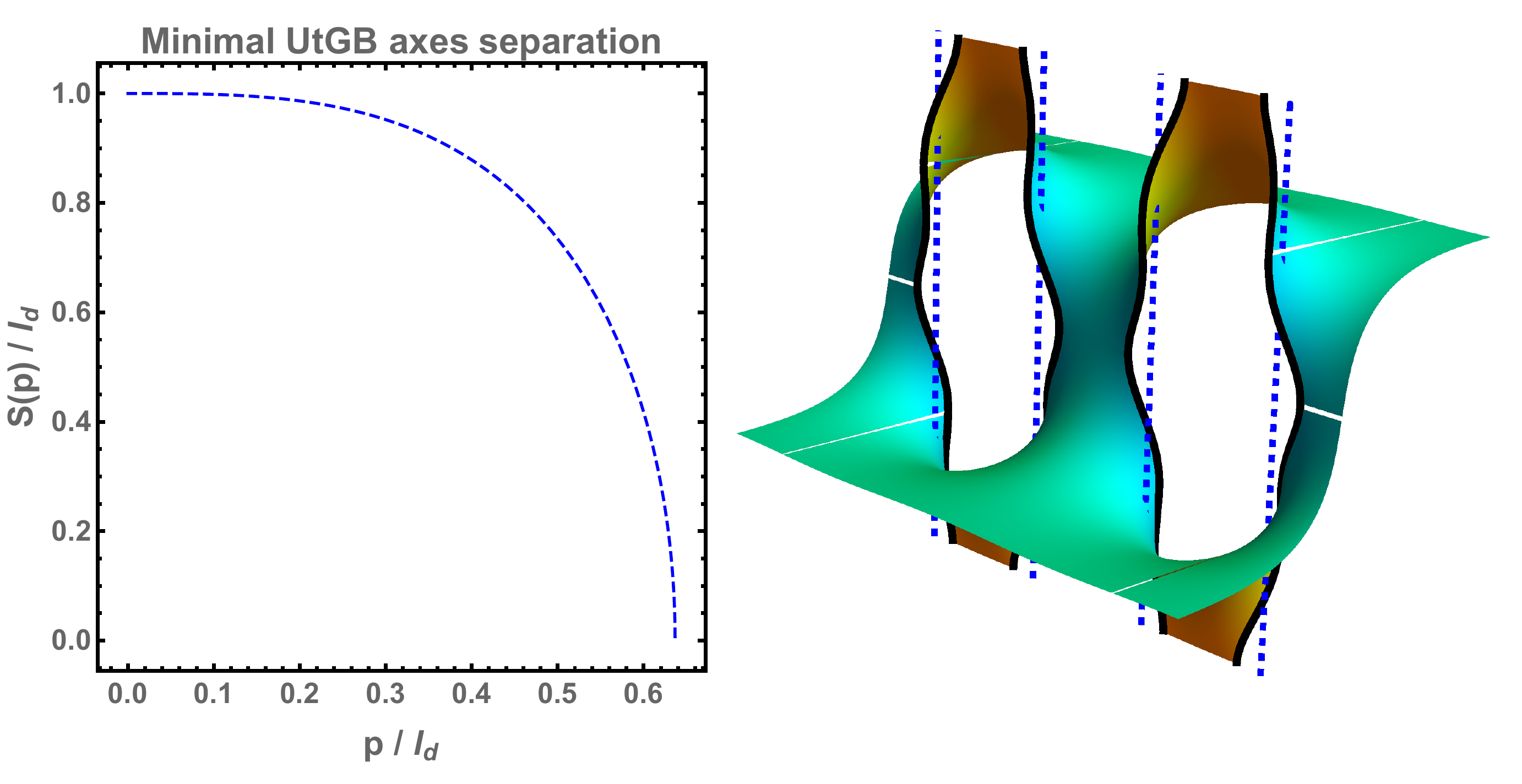}}
    \caption{(Left) Effective axes distance $\ell_{\mathrm{eff}}$ of minimal TGB's as a function of the pitch $p$ and initial axes separation $\ell_d$. For a fixed $\ell_d$, the two axes coincide at the critical value $p_c=\frac{2\ell_d}{\pi}$, i.e., the effective distance $\ell_{\mathrm{eff}}$ vanishes. In addition, the corresponding minimal surface has both asymptotic planes with the same inclination meaning that the unit cell collapses to a plane. The dashed blue lines indicate the position of the axes before the deformation. (Right) Minimal axes separation $S(p)$ of minimal UtGB's as a function of the pitch $p$ and initial axes separation $\ell_d$. On average, the distance between neighboring axes remains constant. However, the top and bottom portions of the two axes in a unit cell approach while for two neighboring cells the approach happens around the center of the corresponding neighboring axes. For a fixed $\ell_d$, neighboring axes touch at the critical value $p_c=\frac{2\ell_d}{\pi}$. The dashed blue lines indicate the position of the axes before the deformation.} 
    \label{fig:effDistAxesTGB}
\end{figure}

{Now, we consider the construction of a minimal UtGB as an Enneper immersion, Fig. \ref{fig:TGB} right.} The corresponding harmonic function is then given by
\begin{equation}
 h(z)=p\,\mbox{Im}\left(\ln\sin\frac{\pi z}{2\ell_d}\right)-p\,\mbox{Im}\left(\ln\cos\frac{\pi z}{2\ell_d}\right),
\end{equation}
where $p$ and $-p$ are the pitches of the helical motifs in each dipole pair that are separated by a distance $\ell_d$. From
\begin{equation}
    \frac{\partial h}{\partial z}=\frac{p\pi}{4\rmi\ell_d}\left[\cot\left(\frac{\pi z}{2\ell_d}\right)+\tan\left(\frac{\pi z}{2\ell_d}\right)\right]=\frac{p\pi}{2\rmi\ell_d}\csc\left(\frac{\pi z}{\ell_d}\right),
\end{equation}
we can compute the auxiliary function $P(z)=\int(\partial_zh)^2\rmd z$ as
\begin{equation}
 P(z)=\frac{p^2\pi}{4\ell_d}\cot\left(\frac{\pi z}{\ell_d}\right)=\sum_{n\in\mathbb{Z}}\frac{p^2}{4(z-n\ell_d)}.
\end{equation}
Contrarily to minimal TGB's, for a minimal UtGB we just have an infinite sum of the functions $P(z)$ associated with helicoids of pitch of magnitude $|p|$ located at $\{n\ell_d:n\in\mathbb{Z}\}$ without any further correction.

Upon application of the deformation in equation (\ref{eq::EnneperDeformation}) leading to the Enneper graph of a minimal UtGB, any pair of neighboring axes {deform} under the increase of $p/\ell_d$, as depicted in Fig. \ref{fig:effDistAxesTGB} (Right). On average, the distance between neighboring axes remains constant. However, the top and bottom portions of the two axes in the same unit cell approach while for two neighboring axes from distinct cells this happens around the center of axes. For a fixed $\ell_d$, neighboring axes finally touch at the critical value $p_c=\frac{2\ell_d}{\pi}$, which corresponds to the topological transition observed in the level curves of $\vert g(z)\vert$ as depicted in Fig. \ref{fig:LevelSetsTGBGaussMap}. The pitch remains unchanged during the deformation process.

From the second derivative $\partial^2_zh=-\frac{p\pi^2}{2\rmi\ell_d^2}\cot(\frac{\pi z}{\ell_d})\csc(\frac{\pi z}{\ell_d})$, the Gaussian curvature is
\begin{equation}
    K = -\frac{p^2\pi^4}{\ell_d^4}\frac{\vert\sin(\frac{\pi z}{\ell_d})\vert^4\vert\cos(\frac{\pi z}{\ell_d})\vert^2}{\left(\vert\sin(\frac{\pi z}{\ell_d})\vert^2+\frac{p^2\pi^2}{4\ell_d^2}\right)^4}.
\end{equation}
Finally, the stereographic projection of the unit normal of the layers of a minimal UtGB is 
\begin{equation}
    g=-\frac{1}{\partial_zh}=\frac{2\ell_d}{p\pi\rmi}\sin\left(\frac{\pi z}{\ell_d}\right).
\end{equation}
The level sets of $\vert g\vert$ help to find the proper domain to compute the Enneper graph of a minimal UtGB (see Fig. \ref{fig:LevelSetsTGBGaussMap}, right). In addition, the image of the surface normal over the fundamental domain $\Omega_0=\{(x,y)\in[0,2\ell_d]\times\mathbb{R}:\vert g(x+\rmi y)\vert\geq1\}$ covers the unit sphere exactly twice, which results in a total curvature of $\int_{\Omega_0}K\rmd A=-4\pi$. For example, one copy comes from $\Omega_0\cap[0,2\ell_d]\times[0,\infty)$ and another from $\Omega_0\cap[0,2\ell_d]\times(-\infty,0]$.

\section{Second variation of the area functional}

The above procedures allow us to construct explicitly exact minimal surfaces with any desired arrangement of helical motifs. The physical motivation for this  construction is that minimal surfaces arise naturally as critical points of both the area of a surface and of the Helfrich free energy. However, while minimal surfaces are stationary points of these functionals, they are not necessarily (local) minima. To establish that a given configuration is indeed of minimal energy the second variation should also be examined. For simplicity, we restrict ourselves to analyzing the area functional alone. Moreover, the configurations we will consider will be restricted to harmonic functions of the form of equation \eqref{eq::hOfManyMotifs}, i.e., arrangements of helical motifs that asymptote to helicoidal surfaces. This is not the most general case. Nonetheless, as will be next demonstrated, this restricted class singles out global pitch balance as a necessary condition for stability. While claims of stability are valid only within this restrictive set of asymptotically helicoidal surfaces, surfaces proved to be locally unstable will remain so even when considering all possible harmonic functions.

Considering normal variations of a minimal surface $\mathbf{r}:D\to \Sigma$ in the direction of $\mathbf{V}=v\mathbf{N}$, where $v\in H(D)=\{f\in C^{\infty}(D):f\vert_{\partial D}=0\}$, the second derivative of the area functional is \cite{BarbosaAJM1976}
\begin{equation}
    I(v)\equiv\frac{\rmd^2}{\rmd t^2}\mbox{Area}\vert_{\{\Sigma:H=0\}} = \int_{\Sigma}v(-\nabla_a^2v+2Kv)\rmd A.
\end{equation}
Therefore, a minimal surface is said to be \emph{stable}, i.e., it is a local minimum, if $I(v)>0$ for all $v\in H(D)$. (When $\bar{D}=D\cup\partial D$ is not compact, the surface is said to be stable if it is stable for all $\Omega\subset D$ such that $\bar{\Omega}$ is compact.) We note that our variations are assumed to vanish on the boundaries (which are also assumed stationary). Thus, while negative values in the second variation imply instability, a positive definite second variation does not necessarily imply stability if the helical motifs are allowed to change their relative orientation or move in space.

A stability criterion for fixed boundaries can be established based on the area over the unit sphere covered by the surface unit normal $\mathbf{N}$. If the area of this spherical image is smaller than $2\pi$, then the corresponding minimal surface has to be stable \cite{BarbosaAJM1976}. (It is worth mentioning that for this stability criterion to work $\mathbf{N}$ does not need to be a one-to-one map.) In our case, the investigation of $g$, the stereographic projection of the surface unit normal $\mathbf{N}$, will play an important role in the study of the stability of minimal helical motifs since the corresponding Gauss map is just the quotient of polynomials. Finally, a criterion of instability that will be useful to our purposes is the following: if the surface normal $\mathbf{N}$ is a one-to-one map from $\Sigma$ to $\mathbb{S}^2$ and $\mathbf{N}(\Sigma)$ is a hemisphere, in particular, the spherical area is $2\pi$, then the corresponding minimal surface is unstable \cite{Rado1933,BarbosaAJM1976}.  (When the area of $\mathbf{N}(\Sigma)\subset\mathbb{S}^2$ is precisely $2\pi$, but $\mathbf{N}$ is not one-to-one or does not cover a hemisphere, the minimal surface may or may not be stable \cite{KoisoJMSJ1984}.)

\subsection{Two equally handed helical motifs are unstable}

The stereographic projection of the Gauss map of a pair of equally handed helical motifs is given by 
$$
g(z)=-\frac{\rmi}{p/R}\left(\frac{z}{R}-\frac{R}{4z}\right)\propto\frac{z}{R}\,\mbox{ if }\,\left\vert\frac{z}{R}\right\vert\gg1.
$$
For any given $w\in\mathbb{C}$, the equation $g(z)=w$ will generically have two distinct solutions and it then follows that the unit normal $\mathbf{N}$ covers the north hemisphere of the unit sphere $\mathbb{S}^2$ exactly twice. (The domain $\Omega_N$ is chosen such that $\mathbf{N}$ takes values on the north hemisphere only.) For the sake of our study of stability, we do not need to count the spherical area with multiplicity and, therefore, the area of the spherical image of any piece $\mathbf{r}:D\subset\Omega_N\to\mathbb{R}^3$ of a minimal pair of equally handed motifs is bounded by $2\pi$. We are going to show that there exists a bounded domain $D_c\subset\Omega_N$ whose spherical image is precisely the north hemisphere and the corresponding normal is a 1-1 map, from which it follows that pairs of equally handed helical motifs are not local minima of the area functional with respect to normal variations that leave the boundary fixed. 

In general, any bounded region $D\subset\Omega_N$ is necessarily contained in a sufficiently large disc $D(0,\rho)\subset\mathbb{C}$, $\rho\gg1$. Under the stereographic projection $\Pi:\mathbb{S}^2\to\mathbb{C}$, a sufficiently small neighborhood of the north pole is mapped on the complement of  $D(0,\rho)$, in particular, it is mapped outside the domain $D$. Now, since $g$ behaves linearly for large values of $\vert z/R\vert$, any sufficiently small neighborhood of the north pole is mapped under $g=\Pi\circ\mathbf{N}$ outside $D\subset D(0,\rho)$ only once! Noticing that $g(z=0)=\infty$, i.e., $\mathbf{N}(0)$ is the north pole of $\mathbb{S}^2$, we conclude that there exists a neighborhood of the north hemisphere which is the image under $g$ of some region inside $D(0,\rho)$. Thus, it follows that there should exist a domain $D_c\subset\Omega_N$ such that its image under $g$ is precisely the north hemisphere, as we are going to show below in more detail.

The norm of $g$ can be computed as
\begin{equation}
     \vert g\vert = \frac{2R}{p}\left\vert \frac{r}{R}\rme^{\rmi\theta}-\frac{\rme^{-\rmi\theta}}{4r/R}\right\vert = \frac{2R}{p}\sqrt{\cos^2\theta+\sinh^2(\ln2\frac{r}{R})}.
\end{equation}
Using that the inverse of the stereographic projection is $\Pi^{-1}(w)=(\frac{2\re(w)}{1+\vert w\vert^2},\frac{2\im(w)}{1+\vert w\vert^2},\frac{-1+\vert w\vert^2}{1+\vert w\vert^2})$, we immediately see that the parallels of the unit sphere, i.e., lines of constant latitude, are associated with the level curves of $\vert g\vert$. We then conclude that the unit normal of a pair of equally handed helical motifs when restricted to $D_c=\{z:r\leq \frac{1}{2}\}\cap\Omega_N$ is a one-to-one map on the north hemisphere of the unit sphere. (The second copy comes from considering the outside of the disc $\{z:r\leq \frac{1}{2}\}$). In conclusion, the normal is a one-to-one map on the north hemisphere when restricted to $D_c$ and the minimal surface corresponding to a pair of equally handed motifs is unstable for any domain $D$ containing $D_c$. (If $D$ were stable, any subdomain $D'\subseteq D$ would have to be stable.)

The proof above is done for the symmetric case, $p_1=p_2$, but similar arguments can be devised to show that any pair of helical motifs such that $p_{1}\ne-p_{2}$ has to be unstable. (In fact, neutral total pitch is a necessary condition for stability, as will be shown in Theorem \ref{Theo::StabParkingGarageIIbkCoeff} below.)

\subsection{Minimal dipoles are stable}

Since the Gauss map of a minimal dipole is a quadratic function on the complex plane, 
$$g(z)=-\frac{2\rmi}{p/R}\left[\left(\frac{z}{R}\right)^2-\frac{1}{4}\right]\propto\left(\frac{z}{R}\right)^2\,\mbox{ if }\,\left\vert\frac{z}{R}\right\vert\gg1,$$ 
for any $w$, $g(z)=w$ generically have two solutions and, as in the previous case, the unit normal $\mathbf{N}$ covers the north hemisphere of $\mathbb{S}^2$ exactly twice. In addition, the area of the spherical image of any piece $\mathbf{r}:D\subset\Omega_N\to\mathbb{R}^3$ of a minimal dipole is bounded by $2\pi$. We are going to show that this area is in fact smaller than $2\pi$ for any bounded domain $D\subset\Omega_N$, from which we will conclude that minimal dipoles are local minima of the area functional for normal variations that leave the boundary fixed. Indeed, first note that any bounded region $D$ of $\Omega_N$ is necessarily contained in a sufficiently large disc $D(0,\rho)$. Since the north pole is mapped under stereographic projection $\Pi:\mathbb{S}^2\to\mathbb{C}$ on infinity, each point on the sphere sufficiently close to the north pole should be mapped on a complex number of sufficiently large modulus. Now, using that $g\propto z^2$ for large values of $\vert z/R\vert$, the two copies associated with a sufficiently small neighborhood of the north pole is necessarily mapped under the Gauss map $g$ outside $D\subseteq D(0,\rho)$. In conclusion, it follows that the spherical image of any bounded $D\subset\Omega_N$ is strictly contained in a hemisphere and, therefore, the corresponding area is smaller than $2\pi$. This can be alternatively confirmed by the fact that the parallels of the unit sphere are associated with the level curves of the Cassini ovals. Indeed, the $z$-coordinate of $\mathbf{N}=\Pi^{-1}\circ g$ is given by $N_3=\frac{-1+\vert g\vert}{1+\vert g\vert}$, which is constant if and only if $\vert g\vert$ is constant and, in addition, $N_3\to1\Leftrightarrow\vert g\vert\to\infty$ (See Fig. \ref{fig:LevelSets2-HelGaussMap}, Left).
  
\subsection{Stability criterion for finitely many helical motifs}
\label{subsect::StabilityOfGarages}

We have seen on the previous subsections that minimal pairs of unequal handed helical motifs, $p_1+p_2\not=0$, are not local minima of the area functional, while minimal dipoles are. Therefore, we arrive at the important conclusion that for two given helical motifs, neutral total pitch is a necessary and sufficient condition for the corresponding minimal surface be a local minimum of the area functional. For more than two helical motifs, extra conditions must be imposed to guarantee stability.

Following Barbosa--Do Carmo \cite{BarbosaAJM1976}, the stability of a minimal surface may be investigated by computing the area under the Gauss map. For $n$ helical motifs glued together, we will proceed as before and study how the complex plane is covered by the stereographic projection of the unit normal, which is a quotient of polynomials of degree $n$ and $n-1$. 

We now formulate a stability criterion for minimal parking garages depending on the position and pitches of the helical motifs. A proof for the first theorem below is based on the observation that $n$ helical motifs are stable if only if $g$, which in general is a rational function, is a polynomial of degree $n$. The second theorem is a reformulation of the first and will allow us to check the stability of minimal helical motifs with respect to normal variations that leave the boundary fixed more easily.

\begin{theorem}[Stability of minimal helical motifs. I]\label{Theo::StabParkingGarageIbkCoeff}
Consider a set of $n$ helical motifs of pitches $p_1,\dots,p_n\in\mathbb{R}$ located at $z_1,\dots,z_n\in\mathbb{C}$. The corresponding asymptotically flat Enneper graph is a stable minimal surface with respect to normal variations that leave the boundary fixed if and only if
\begin{equation}
\forall\,k\in\{0,\dots,n-2\},\,\left\{
\begin{array}{ccc}
d_{k+1} & \equiv & \displaystyle\sum_{j=1}^n p_js_k(z_1,\dots,z_{j-1},z_{j+1},\dots,z_n) = 0\\
d_n & \equiv & \displaystyle\sum_{j=1}^n p_js_{n-1}(z_1,\dots,z_{j-1},z_{j+1},\dots,z_n)\not=0\\
\end{array}
\right.,
\end{equation}
where $s_k(\cdot,\dots,\cdot)$ is the elementary symmetric polynomial, i.e., for any $(x_1,\dots,x_m)$, $s_0(x_1,\dots,x_m)=1$ and $s_k(x_1,\dots,x_m)=\sum_{1\leq j_1<\dots<j_k\leq m}\prod_{j=1}^kx_j$. 
\end{theorem}
\begin{proof}
The stereographic projection of the unit normal $\mathbf{N}$ of the Enneper graph $\Sigma$ of $n$ helical motifs is
\[
 g = -\left(\displaystyle\frac{1}{2\rmi}\sum_{k=1}^n\frac{p_k}{z-z_k}\right)^{-1}= -2\rmi\,\frac{z^n-a_1z^{n-1}+a_2z^{n-2}\dots+(-1)^na_n}{d_1z^{n-1}-d_2z^{n-2}+\cdots+(-1)^{n-1}d_{n}}, 
\]
 where the coefficients $a_k$ and $d_k$ are computed from $\{p_i\}_{k=1}^n$ and $\{z_i\}_{k=1}^n$ in terms of the symmetric polynomials as $ a_k=s_k(z_1,\dots,z_n)$ and $d_k=\sum_{j=1}^n p_js_{k-1}(z_1,\dots,z_{j-1},z_{j+1},\dots,z_n)$.

Given $w\in\mathbb{C}$, the equation $g(z)=w$ can be rewritten as a polynomial equation of degree $n$ whose coefficients depend on $w$, $\{p_j\}$, and $\{z_j\}$:
$$g(z)=w\Leftrightarrow z^n-c_1z^{n-1}+\dots+(-1)^nc_n=0,\,c_k=c_k(w,\{p_j\},\{z_j\})=a_k+\rmi w\frac{d_k}{2}.$$
Thus, $g(z)=w$ generically has $n$ solutions and, consequently, the Gauss map $g$ should cover the north hemisphere of $\mathbb{S}^2$ exactly $n$ times. (The domain of the Enneper graph of helical motifs was defined as $\Omega_N=\{z\in\mathbb{C}: N_3>0,\,\mathbf{N}(z)=(N_1(z),N_2(z),N_3(z))\}$.)
 
On the one hand, if the denominator of $g$ is constant, i.e., $d_1=\dots=d_{n-1}=0$, then $g$ behaves as $\pm\frac{2\rmi}{d_n}z^n$ for large values of $\vert z/R\vert$ and the $n$ copies associated with a sufficiently small neighborhood of the north pole is necessarily mapped under the Gauss map $g$ outside a disc $D(0,\rho)$, for some $\rho\gg1$. It follows that any bounded subdomain of $\Omega_N$ has an area  on the unit sphere under $\mathbf{N}$ smaller than $2\pi$ and, consequently, $\Sigma$ is stable.

On the other hand, if the denominator of $g$ is not constant, i.e., there exists some $k_0<n$ such that $d_{k_0}\not=0$, then $g$ behaves as $-\frac{2\rmi}{d_{k_0}}z^{n-k_0}$ for large values of $\vert z/R\vert$, say for all $z\in D(0,\rho)$, $\rho\gg1$. This means that a sufficiently small neighborhood of the north pole is covered under $\mathbf{N}$ only $n-k_0$ times. Therefore, the remaining $k_0$ copies of a neighborhood of the north pole should come from the inside of the disc $D(0,\rho)$. It then follows that there must exist a bounded subdomain $D_c$ of $\Omega_N$ such that $\mathbf{N}$ is 1-1 over $D_c$ and $\mathbf{N}(D_c)\subset\mathbb{S}^2$ is precisely the north hemisphere. Consequently, $\Sigma$ is unstable.
\end{proof}

The coefficients $d_{k+1}$ can be rewritten in a more convenient form. For $k=0$, we obtain the total pitch 
\begin{equation}
d_1=\sum_{j=1}^n p_js_0(z_1,\dots,z_{j-1},z_{j+1},\dots,z_n)=\sum_{j=1}^np_j.
\end{equation}
For $k=1$, the coefficient $d_2$ can be rewritten as a function of the "center of mass" of $z_1,\dots,z_n$:
\begin{equation}
  d_2= \sum p_js_1(z_1,\dots,z_{j-1},z_j,z_{j+1},\dots,z_n)-\sum p_jz_j= d_1s_1(\mathbf{z})-\sum p_jz_j,
\end{equation}
where we adopted the shorthand notation $s_k(\mathbf{z})\equiv s_k(z_1,\dots,z_n)$. In general, we can write any $d_k$ as function of $b_k\equiv\sum_{j=1}^np_jz_j^k$. (Notice that $b_0=d_1=\sum_jp_j$.) Indeed, we can rewrite $d_{k+1}$ as
\begin{eqnarray}
  d_{k+1}  & = & \sum_j p_js_k(\dots,z_j,\dots)-\sum_j p_j\sum_{j_1<\dots<j_{k-1};j_m\not=j}z_jz_{j_1}\cdots z_{j_{k-1}}\nonumber\\
  & = & d_1s_k(\mathbf{z})-\sum_j p_jz_j\sum_{j_1<\dots<j_{k-1};j_m\not=j}z_{j_1}\cdots z_{j_{k-1}}.
\end{eqnarray}
The second term in the last equality behaves as the coefficient $d_k$ associated with the (complex) pitches $p_1z_1,\dots,p_nz_n$. Thus, proceeding inductively,
\begin{eqnarray}
  d_{k+1} & = & d_1s_k(\mathbf{z})-\sum_j p_jz_j\sum_{j_1<\dots<j_{k-1};j_m\not=j}z_{j_1}\cdots z_{j_{k-1}}\nonumber\\
  & = & d_1s_k(\mathbf{z})-\left[(\sum_jp_jz_j)s_{k-1}(\mathbf{z})-\sum_j p_jz_j^2\sum_{j_1<\dots<j_{k-2};j_m\not=j}z_{j_1}\cdots z_{j_{k-2}}\right]\nonumber\\
  & = & \dots\,\,\,\, = \sum_{j=0}^k(-1)^jb_js_{k-j}(\mathbf{z}).
\end{eqnarray}

In conclusion, theorem \ref{Theo::StabParkingGarageIbkCoeff} can be reformulated as follows.

\begin{theorem}[Stability of minimal helical motifs. II]\label{Theo::StabParkingGarageIIbkCoeff}
Consider a set of $n$ helical motifs of pitches $p_1,\dots,p_n\in\mathbb{R}$ located at $z_1,\dots,z_n\in\mathbb{C}$. The corresponding Enneper graph is a stable minimal immersion with respect to normal variations that leave the boundary fixed if and only if
\begin{equation}
b_0\equiv\sum_{j=1}^np_j=0\mbox{ and }b_{k} \equiv \displaystyle\sum_{j=1}^n p_jz_j^k = 0,\mbox{but }
b_{n-1} \equiv  \displaystyle\sum_{j=1}^n p_jz_j^{n-1}\not=0,
\end{equation}
where $k\in\{1,\dots,n-2\}$.
\end{theorem}
\begin{proof}
The coefficients $d_{k+1}$ and $b_k$ are related by $d_{k+1}=\sum_{j=0}^k(-1)^jb_js_{k-j}(\mathbf{z})$. In addition, this correspondence can be rewritten as a system of linear equations:
\begin{equation}
    \left(
    \begin{array}{ccccc}
        1    &  0     & 0     & \cdots & 0\\
        s_1    & -1   & 0     & \cdots & 0\\
        s_2    & -s_1   & 1   & \cdots & 0\\
        \vdots &\vdots  &\vdots & \cdots & 0\\
        s_{n-1}&-s_{n-2}&s_{n-3}& \cdots & (-1)^{n-1}\\
    \end{array}
    \right)
    \left(
    \begin{array}{c}
         b_0  \\
         b_1 \\
         b_2\\
         \vdots\\
         b_{n-1}\\ 
    \end{array}
    \right)    = \left(
    \begin{array}{c}
         d_1  \\
         d_2 \\
         d_3\\
         \vdots\\
         d_{n}\\ 
    \end{array}
    \right).
\end{equation}
Since the determinant of the coefficient matrix of the system above does not vanish, $(-1)^n\not=0$, it follows that $\{b_k\}$ can be uniquely computed for a given $\{d_k\}$ and vice-versa. In particular, $(d_1,\dots,d_{n})=(0,\dots,0,d_n\not=0)$ if and only if $(b_0,\dots,b_{n-1})=(0,\dots,0,b_{n-1}\not=0)$.
\end{proof}

{It should be remarked that variations that vanish on the boundaries constitute a subset of those variations that are allowed to move and/or distort the boundaries. Therefore, as generic boundary variations will broaden the class of perturbations, they may render surfaces that were stable under the Barbosa-Do Carmo theorem, unstable. On the other hand, unstable surfaces must remain unstable. In short, it follows as a corollary that the stability conditions stated in Theorem \ref{Theo::StabParkingGarageIbkCoeff} or Theorem \ref{Theo::StabParkingGarageIIbkCoeff} will become \emph{necessary conditions} for stability with respect to variations that do not have to vanish on the boundary.}

\subsubsection{Example of stable pitch balanced helical motifs}

We demonstrated that two helical motifs forming a minimal dipole is an example of a stable Enneper graph. 
We now provide another example consisting of a helical motif of pitch $-np$ balanced by $n$ helical motifs of pitch $p$ each. It is worth mentioning that this type of configuration resembles the basic units composing the helical geometry in plant thylakoids \cite{BussiPNAS2019}, where the central right-handed helical motif is balanced by several smaller left-handed motifs. 

Let $p_0=-np$, $p_1=p,\dots,p_{n-1}=p$, and $p_n=p$ be the pitches of $n+1$ helical motifs located at the center and vertices of a regular $n$-gon, respectively. The first equilibrium equation is naturally satisfied $b_0=\sum_{j=0}^np_j=0$. For the remaining equations, we may use the following identity for the $n$-th roots of unity $\zeta_0=1,\zeta_1,\dots,\zeta_{n-1}$:
    $\sum_{j=0}^{n-1}\zeta_j^k=0$ if $k$ is not a multiple of $n$, but  $\sum_{j=0}^{n-1}\zeta_j^k=n$ if otherwise \cite{Ahlfors1979}. Now, writing $z_0=0$ and $z_j=R\,\zeta_{j-1}$, $j=1,\dots,n$, it follows that 
\begin{equation}
 b_k = \sum_{i=0}^np_iz_i^k=pR^k\sum_{i=0}^{n-1}\zeta_i^k=\left\{
    \begin{array}{ccc}
        0 & \mbox{ if } & k<n  \\
        npR^n & \mbox{ if } & k=n\\ 
    \end{array}
    \right..
\end{equation}   
From Theorem \ref{Theo::StabParkingGarageIIbkCoeff}, we conclude that the configuration of helical motifs $\{(p_0=-np,z_0=0)\}\cup\{(p_j=p,z_j=R\zeta_{j-1})\}_{j=1}^{n}$, where $p\not=0$ and $R>0$, is a stable minimal Enneper immersion.

\section{Discussion}
Examining the non-linear geometric interactions between helical motifs in lamellar structures is necessary for understanding the inter-motif interaction at distances comparable to their pitch. Recent experiments revealed such closely packed helical motifs in the Endoplasmic reticulum \cite{TerasakiCell2013} and plant thylakoids \cite{BussiPNAS2019}.
By providing a constructive method for producing exact minimal surfaces from harmonic functions we are able to examine the geometry and interactions between such closely spaced helical motifs, which are not amenable to a small-slope approximation. 

Focusing on asymptotically helicoidal surfaces we proved that for the obtained minimal surfaces to be locally stable minima of the area functional the embedded helical motifs must be pitch balanced. Note, however, that when considering compact domains with boundaries (as is the case for the pitch balanced helical motifs in the endoplasmic reticulum and plant thylakoids) the surfaces may not be asymptotically helicoidal, and thus non-pitch balanced helical motifs may form a stable configuration. Moreover, the stability for pitch balanced arrangements of motifs is obtained by considering perturbations that vanish on the boundaries, and in particular fix the helical motifs in place. In many physical and biological cases the helical motifs are mobile and can vary their relative orientation and position to further lower their energy. {Additionally, this also opens the problem of considering an energy contribution from the boundary, as for example, the Euler-Plateau problem where one considers the boundary as an Euler elastica \cite{GiomiPRSA2012} or the Kirchhoff-Plateau problem where one considers the boundary as a Kirchhoff rod \cite{GiusteriJNS2017}.} The systematic study of these relative forces and torques between motifs {as well as the energy contribution of the boundaries are}  left to future work. Extending the results of stability also requires that we resolve the core of each motif. The Enneper immersion yields in many cases not only the desired minimal surface and helical motif, but also a non-physical section of minimal surface to accommodate ``excess material''. In the present work, we chose the Gauss Map to define the boundaries between the physical and non-physical portions of the Enneper immersion. Other choices may also be possible and are expected to vary the inter-motif interactions.

{Under the deformation leading to a minimal Enneper immersion, the distance between neighboring layers measured with respect to the vertical ($z$-axis) remains unchanged. The same, however, can not be said of the distance between layers as measured in the normal direction for the case of finitely many motifs.  In this work, we focused on producing exact minimal surfaces with desired topology, and thus we were not directly controlling the inter-layer separation, nor were we considering the elastic penalization associated with deviations from such constant spacing. While for some applications, such as smectic liquid crystals such terms must be take into account, it is not clear whether this is the case for biological membranes such as the thylakoid. }

\begin{figure}[t]
    \centering
    \includegraphics[width=\linewidth]{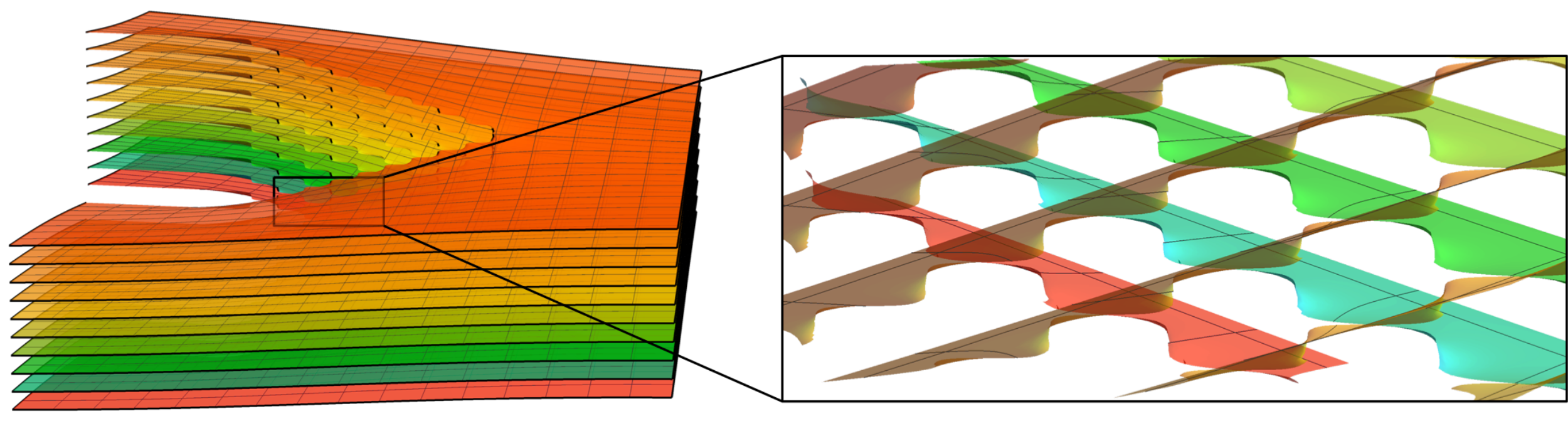}
    \caption{Multiple layers of minimal surfaces in which 10 identical helical motifs of pitch $p$ are embedded. Visualized far from the helical cores, the structure looks like interlaced helicoids of pitch $P=10 p$. Near the cores, however, the structure resembles that of a twist grain boundary.}
    \label{fig:my_label}
\end{figure}

The constructive recipe provided here allows us to examine arbitrary arrangements of helical motifs at any separation and at any scale. For example, resolving the near field for a finite chain of closely packed identical motifs allows us to show how the expected TGB-like structure of the near field reconciles with the predicted far field helicoidal structure, as observed in Figure \ref{fig:my_label}. 

We proved that every harmonic function on the complex plane can be distorted to an exact minimal surface using lateral displacements alone, and that every minimal surface may be obtained through this recipe;
For any given minimal surface we can find a harmonic function whose Enneper immersion yields the desired minimal surface. This invertibility renders this method a valuable tool with which to study quantitatively the non-linear interaction between motifs in natural biological and physical settings. The restriction to asymptotically flat surfaces arises from considering only a subset of harmonic functions, and is not an inherent limitation of the method. Extending the analysis presented here to account for all harmonic functions will yield an explicitly controllable and exhaustive representation of minimal surfaces.

\vskip6pt

\appendix

\section*{Appendix}

\section{Holomorphic representation of minimal surfaces}
\label{AppHolomRepMinSurf}

As it is well known, minimal surfaces in Euclidean space admit a representation in terms of holomorphic functions \cite{Barbosa1986,Weierstrass}. In this Appendix, we provide a self-contained discussion of this topic and the relation between the Enneper and Weierstrass representations.

Let $\mathbf{r}(z)=(f_1(z),f_2(z),f_3(z))$ be a smooth parametrization of a regular minimal surface, where $f_j$ is a smooth real function. Define the following complex functions $\phi_j=\frac{\partial f_j}{\partial x}-\rmi \frac{\partial f_j}{\partial y}$, $j=1,2,3$. From the relation $\phi_1^2+\phi_2^2+\phi_3^2=\partial_x\mathbf{r}\cdot\partial_x\mathbf{r}-\partial_y\mathbf{r}\cdot\partial_y\mathbf{r}-\rmi\partial_x\mathbf{r}\cdot\partial_y\mathbf{r}=a_{11}-a_{22}-\rmi a_{12}$, it follows that $\mathbf{r}$ is conformal if, and only if, $\sum_j\phi_j^2=0$. In addition, $\sum_j\vert\phi_j\vert^2=\partial_x\mathbf{r}\cdot\partial_x\mathbf{r}+\partial_y\mathbf{r}\cdot\partial_y\mathbf{r}=2a_{11}$ and, therefore, $\mathbf{r}$ is regular if and only if $\sum_j\vert\phi_j\vert^2\not=0$. Now, since $\mathbf{r}$ is conformal and minimal, the second derivatives must satisfy $\partial_x^2\mathbf{r}=-\partial_y^2\mathbf{r}$ and, from the commutativity of mixed derivatives, we also have $\partial_{xy}^2f_j=\partial_{yx}^2f_j$. From the definition of $f_j$, these conditions on the second derivatives give $\partial_x\re(\phi_j)=\partial_y\im(\phi_j)$ and $\partial_y\re(\phi_j)=-\partial_x\im(\phi_j)$. In other words, $\phi_1$, $\phi_2$, and $\phi_3$  are holomorphic and, consequently, we can write $\mathbf{r}=\re(\int\phi_1\rmd z,\int\phi_2\rmd z,\int\phi_3\rmd z)$.

Conversely, let a regular surface $\Sigma$ be parametrized by $\mathbf{r}=\re(\int^z\phi_1\rmd\zeta,\int^z\phi_2\rmd\zeta,\int^z\phi_3\rmd\zeta)=\int^z\frac{\phi+\bar{\phi}}{2}\rmd\zeta$, where $\phi_j$ ($j=1,2,3$) is holomorphic and $\phi=(\phi_1,\phi_2,\phi_3)$. Employing the Wirtinger derivatives \cite{Ahlfors1979}: $\pd_{x}=\pd_{z}+\pd_{\zb}$ and $\pd_{y}=\rmi(\pd_{z}-\pd_{\zb})$, it is immediate to obtain that the first and second derivatives of $\mathbf{r}$ are
\begin{equation}\label{Eq::VelandAccelMinimalImmersionAsIsotropicCurve}
\partial_x\mathbf{r}=\frac{\phi+\bar{\phi}}{2},\partial_y\mathbf{r}=-\frac{\phi-\bar{\phi}}{2\rmi}
\mbox{ and }
\partial_{x}^2\mathbf{r}=\frac{\phi'+\bar{\phi}'}{2}=-\partial_y^2\mathbf{r},\partial_{xy}^2\mathbf{r}=-\frac{\phi'-\bar{\phi}'}{2\rmi}.
\end{equation}
Then, the coefficients of the metric, $a_{11}=\partial_x\mathbf{r}\cdot\partial_x\mathbf{r}$, $a_{12}=\partial_x\mathbf{r}\cdot\partial_y\mathbf{r}$, and $a_{22}=\partial_y\mathbf{r}\cdot\partial_y\mathbf{r}$ are 
\[
a_{11}=\frac{1}{4}\sum_{j=1}^3(\phi_j^2+\bar{\phi}_j^2+2\vert\phi_j\vert^2),\,a_{12}=\frac{1}{4\rmi}\sum_{j=1}^3(\bar{\phi}_j^2-\phi_j^2),a_{22}=-\frac{1}{4}\sum_{j=1}^3(\phi_j^2+\bar{\phi}_j^2-2\vert\phi_j\vert^2).
\]
It follows that $\mathbf{r}$ is conformal, i.e., $a_{11}=a_{22}$ and $a_{12}=0$, if and only if $\phi_1^2+\phi_2^2+\phi_3^2=0$. In addition,  $\mathbf{r}$ is regular if and only if $\vert\phi_1\vert^2+\vert\phi_2\vert^2+\vert\phi_3\vert^2\not=0$. Now, noticing that $b_{11}=\partial_x^2\mathbf{r}\cdot\mathbf{N}=-\partial_y^2\mathbf{r}\cdot\mathbf{N}=-b_{22}$, the mean curvature of any conformal immersion $\mathbf{r}=\int\frac{\phi+\bar{\phi}}{2}\rmd z$ vanishes. 

For an Enneper immersion $\mathbf{r}=(L-\bar{P},h)=\re(\int(L'-P')\rmd z,\int-\rmi(L'+P')\rmd z,\int2h_z\rmd z)$, it follows that the metric of the surface given in \eqref{eq::EnneperDeformation} takes the form:
\[
a=\mymat{\pd_{x}\vr\cdot\pd_{x}\vr&\pd_{x}\vr\cdot\pd_{y}\vr
\\
\pd_{x}\vr\cdot\pd_{y}\vr&\pd_{y}\vr\cdot\pd_{y}\vr} = (1+h_{z} h_{\zb})^{2}\mymat{1&0\\0&1}.
\]
It is immediate to verify that the conditions $L'P'=(h_z)^2$ and $\vert L'\vert+\vert P'\vert\not=0$ guarantee that any Enneper immersion is a regular conformal  minimal immersion. Alternatively, to prove the minimality, we may use the fact that for a conformal metric $a_{ij}=F^2\delta_{ij}$, the  Laplacian operator corresponding to $a_{ij}$ is $\nabla^2_a=\frac{1}{F^2}\nabla^2=\frac{4}{F^2}\partial_{\bar{z}}\partial_z$. Now, using that $2H\,\mathbf{N}=\nabla^2_a\mathbf{r}$, the mean curvature vector of an Enneper immersion is
\[
\mathbf{H}\propto (\partial_z\partial_{\bar{z}} \,z-\partial_{\bar{z}}\partial_z\bar{P},\partial_{\bar{z}}\partial_zh)=(0,0,0),
\]
which implies $H=0$, i.e., an Enneper immersion is minimal.

Before we conclude by computing the Gaussian curvature, it is worthwhile to identify the mapping between the present formulation and the well-known Weierstrass representation \cite{Kuhnel2015} in which a minimal surface is expressed in terms of two holomorphic functions $f$ and $g$ through\footnote{Please, note that in Ref. \cite{AndradeJAM1998} it is used a slightly distinct Weierstrass representation for minimal surfaces.} 
\begin{equation}
    \mathbf{r} = \mbox{Re}\left(\int_{z_0}^z\frac{1}{2}f(\zeta)(1-g^2(\zeta))\rmd\zeta,\int_{z_0}^z\frac{\rmi}{2}f(\zeta)(1+g^2(\zeta))\rmd\zeta,\int_{z_0}^zf(\zeta)g(\zeta)\rmd\zeta\right).
\label{eq::weirstrass}    
\end{equation}
Comparing \eqref{eq::weirstrass} and $\mathbf{r}=(z-\bar{P},h)=\re(\int(1-P')\rmd z,\int-\rmi(1+P')\rmd z,\int2h_z\rmd z)$, we identify 
\[
1-P' = \frac{f}{2}(1-g^2),\,-\rmi (1+P') = \frac{\rmi f}{2}(1+g^2),\mbox{ and }2\,\partial_z h=fg.
\] 
from which follows that $g=-\frac{1}{\partial_zh}$  and $ f=-2(\partial_zh)^2$. While the relation between the Weierstrass and Enneper representations is difficult to directly interpret, it allows us to exploit known results from one representation to the other. In particular, the known expression for the Gaussian curvature $K$ in terms of the Weierstrass data \cite{Barbosa1986} yields the  curvature of an Enneper graph through
\begin{equation}
    K=-\left[\frac{4\vert g'\vert}{\vert f\vert(1+\vert g\vert^2)^2}\right]^2=-  4\frac{\vert \partial_{zz} h\vert^2}{(1+\vert \partial_zh\vert^{2})^{4}}=\frac{h_{xx}h_{yy}-h_{xy}^2}{[1+\frac{1}{4}(h_x^2+h_y^2)]^{4}},
\end{equation}
where we used that $\vert P'\vert=\frac{1}{4}\Vert\nabla h\Vert^2$ and $\vert\partial_{zz}h\vert^2=-\frac{1}{4}(h_{xx}h_{yy}-h_{xy}^2)$.

To finish this discussion of holomorphic representation, let us find the geometric interpretation of $g=-1/h_z$, which plays a fundamental role in this text. By seeing the normal $\mathbf{N}$ as a map from $\Sigma$ to the unit sphere $\mathbb{S}^2$, $g$ can be geometrically interpreted as the stereographic projection of $\mathbf{N}$. Indeed, using the tangent vectors of $\mathbf{r}=\re(\int^z\phi_1\rmd\zeta,\int^z\phi_2\rmd\zeta,\int^z\phi_3\rmd\zeta)$ given in equation \eqref{Eq::VelandAccelMinimalImmersionAsIsotropicCurve}, we have
\begin{equation}
\mathbf{N}=\frac{\partial_x\mathbf{r}\times\partial_y\mathbf{r}}{\Vert\partial_x\mathbf{r}\times\partial_y\mathbf{r}\Vert}=\frac{2}{\vert\phi_1\vert^2+\vert\phi_2\vert^2+\vert\phi_3\vert^2}(\im(\phi_2\bar{\phi}_3),-\im(\phi_1\bar{\phi}_3),\im(\phi_1\bar{\phi}_2)).
\end{equation}
Now, substituting $\phi_1=\frac{f}{2}(1-g^2)$, $\phi_2=\frac{\rmi f}{2}(1+g^2)$, and $\phi_3=fg$, and computing the stereographic projection of the unit normal $\mathbf{N}=(N_1,N_2,N_3)$, $\Pi\circ\mathbf{N}=(\frac{N_1}{1-N_3},\frac{N_2}{1-N_3})$, gives the desired identity $g=\Pi\circ\mathbf{N}$.

\enlargethispage{20pt}

\aucontribute{LCBdS and EE conceived the research, carried out the research and drafted the manuscript. LCBdS prepared the figures. Both authors gave final approval for submission.}

\competing{We declare we have no competing interests.}

\funding{This work was funded by the Israel Science Foundation Grant No. 1479/16. E.E. thanks the Ascher foundation for their support. LCBdS acknowledges the support provided by the Mor\'a Miriam Rozen Gerber Fellowship for Brazilian Postdocs.}

\ack{We thank Y. Bussi and Z. Reich for helpful discussions and for introducing us to the fascinating and elaborate structure of plant thylakoids.}


\begin{thebibliography}{10}

\bibitem{Rado1933}
Rad{\'o} T. 1933 {\em On the Problem of {P}lateau}. Berlin: Springer. (\url{https://doi.org/10.1007/978-3-642-99118-9})

\bibitem{HelfrichZfN1973}
Helfrich W. 1973 Elastic properties of lipid bilayers: theory and possible
  experiments. {\em Z. Naturforsch. C} {\bf28}, 693--703. (\url{https://doi.org/10.1515/znc-1973-11-1209})

\bibitem{TerasakiCell2013}
Terasaki M, Shemesh T, Kasthuri N, Klemm RW, Schalek R, Hayworth KJ,
  Hand AR, Yankova M, Huber G, Lichtman JW, Rapoport TA,
  Kozlov MM. 2013 Stacked endoplasmic reticulum sheets are connected by helicoidal  membrane motifs, {\em Cell} {\bf154}, 285--296. (\url{https://dx.doi.org/10.1016/j.cell.2013.06.031})

\bibitem{MustardyPC2008}
Must{\'a}rdy L, Buttle K, Steinbach G, Garab G. 2008 The three-dimensional
  network of the thylakoid membranes in plants: quasihelical model of the
  granum-stroma assembly. {\em Plant Cell} {\bf20}, 2552--2557. (\url{https://doi.org/10.1105/tpc.108.059147})

\bibitem{LibertonPP2011}
Liberton M, Austin JR, Berg RH, Pakrasi HB. 2011 Unique thylakoid
  membrane architecture of a unicellular {N}$_2$-fixing cyanobacterium revealed  by electron tomography. {\em Plant Physiol.} {\bf155}, 1656--1666. (\url{https://doi.org/10.1104/pp.110.165332})

\bibitem{BussiPNAS2019}
Bussi Y, Shimoni E, Weiner A, Kapon R, Charuvi D, Nevo R, Efrati E, Reich Z. 2019 Fundamental helical geometry consolidates the plant
  photosynthetic membrane. {\em Proc. Natl. Acad. Sci.} {\bf116},
  22366--22375. (\url{https://doi.org/10.1073/pnas.1905994116})

\bibitem{KamienPRL1999}
Kamien RD, Lubensky TC. 1999 Minimal surfaces, screw dislocations, and
  twist grain boundaries. {\em Phys. Rev. Lett.} {\bf82}, 2892. (\url{https://doi.org/10.1103/PhysRevLett.82.2892})

\bibitem{SantangeloPRL2006}
Santangelo CD, Kamien RD. 2006 Elliptic phases: a study of the nonlinear
  elasticity of twist-grain boundaries. {\em Phys. Rev. Lett.} {\bf96},
  137801. (\url{https://doi.org/10.1103/PhysRevLett.96.137801})

\bibitem{MatsumotoIF2017}
Matsumoto EA, Kamien RD, Alexander GP. 2017 Straight round the twist:
  frustration and chirality in smectics-{A}. {\em Interface Focus} {\bf7},
  20160118. (\url{https://doi.org/10.1098/rsfs.2016.0118})

\bibitem{GuvenPRL2014}
Guven J, Huber G, Valencia DM. 2014 Terasaki spiral ramps in the rough
  endoplasmic reticulum. {\em Phys. Rev. Lett.} {\bf113}, 188101. (\url{https://doi.org/10.1103/PhysRevLett.113.188101})

\bibitem{AndradeJAM1998}
Andrade P. 1998 Enneper immersions. {\em J. Anal. Math.} {\bf75}, 121--134. (\url{https://doi.org/10.1007/BF02788695})

\bibitem{JoslinMP1983multipole}
Joslin CG, Gray CG. 1983. Multipole expansions in two dimensions. {\em
  Mol. Phys.} {\bf50}, 329--345. (\url{https://doi.org/10.1080/00268978300102381})

\bibitem{RennPRA1988}
Renn SR, Lubensky TC. 1988 Abrikosov dislocation lattice in a model of
  the cholesteric to smectic-{A} transition. {\em Phys. Rev. A} {\bf38}, 2132. (\url{https://doi.org/10.1103/PhysRevA.38.2132})

\bibitem{KamienAML2001}
{{Kamien} RD. 2001 Decomposition of the height function of Scherk's first surface. {\em Appl. Math. Lett.} {\bf14}, 797--800. (\url{https://doi.org/10.1016/S0893-9659(01)00046-5})}


\bibitem{Lazard-HollyIM2001}
{Lazard-Holly H, and Meeks W. 2001 Classification of doubly-periodic minimal surfaces of genus zero, {\em Invent. Math.} {\bf143}, 1--27. (\url{https://doi.org/10.1007/PL00005796})}


\bibitem{BarbosaAJM1976}
{Barbosa} JL, {do Carmo} M. 1976 On the size of a stable minimal surface
  in ${R}^3$. {\em Am. J. Math.} {\bf98}, 515--528. (\url{https://doi.org/10.2307/2373899})

\bibitem{KoisoJMSJ1984}
Koiso M. 1984 On the stability of minimal surfaces in ${R}^3$. {\em J. Math.  Soc. Japan} {\bf36}, 523--541. (\url{https://doi.org/10.2969/jmsj/03630523})

\bibitem{Ahlfors1979}
Ahlfors LV. 1979 {\em Complex analysis: an introduction to the theory of analytic  functions of one complex variable}. McGraw-Hill.

\bibitem{GiomiPRSA2012}
{Giomi L, Mahadevan L. 2012 Minimal surfaces bounded by elastic lines. {\em Proc. R. Soc. A} {\bf468}, 1851--1864. (\url{https://doi.org/10.1098/rspa.2011.0627})}

\bibitem{GiusteriJNS2017}
{Giusteri GG, Lussardi L, Fried E. 2017 Solution of the Kirchhoff--Plateau Problem. {\em J. Nonlinear Sci.} \textbf{27}, 1043--1063. (\url{https://doi.org/10.1007/s00332-017-9359-4})}

\bibitem{Barbosa1986}
{Barbosa} JLM, Colares AG. 1986 {\em Minimal surfaces in
  $\mathbb{R}^3$}. Berlin Heidelberg: Springer. (\url{https://doi.org/10.1007/BFb0077105})

\bibitem{Weierstrass}
Weierstrass K. 2013 {\em Untersuchungen \"uber die {F}l\"achen, deren mittlere {K}r\"ummung \"uberall gleich {N}ull ist}. In: Mathematische Werke vol.~3, pp.~39--52. Cambridge University Press. (Reprinted from: Monatsber. K\"onigl. Preuss. Akad. Wiss. Berlin, 1866, pp. 612--625). (\url{https://doi.org/10.1017/CBO9781139567886.003})

\bibitem{Kuhnel2015}
K{\"u}hnel W. 2015 {\em Differential geometry: curves - surfaces - manifolds}. American Mathematical Society.

\end{thebibliography}
\end{document}